\documentclass{article}

\usepackage{amsmath}
\usepackage{amsthm}
\usepackage{amsfonts}
\usepackage{amsbsy}
\usepackage{amssymb}
\usepackage{pstricks}
\newtheorem{theorem}{Theorem}
\newtheorem{proposition}{Proposition}

\newtheorem{lemma}{Lemma}

\newcommand{\R}{{\mathbb R}}
\newcommand{\Z}{{\mathbb Z}}

\newcommand{\set}[2]{ \left\{ #1 \ \left| \ #2 \right. \right\} }

\newcommand{\real}{\mathop{\mathrm{Re}}}
\newcommand{\imaginary}{\mathop{\mathrm{Im}}}

\newcommand{\Qj}{Q_{ \! \! j}}

\newcommand{\scaled}{{\cal S}}
\newcommand{\one}{\mathbf 1}
\title{Rank and regularity for averages over submanifolds}
\author{Philip T. Gressman\footnote{Partially supported by NSF grant DMS-0653755.}}
\begin{document}
\maketitle
\begin{abstract}
This paper establishes endpoint $L^p-L^q$ and Sobolev mapping properties of Radon-like operators which satisfy a homogeneity condition (similar to semiquasihomogeneity) and a condition on the rank of a matrix related to rotational curvature.  For highly degenerate operators, the rank condition is generically satisfied for algebraic reasons, similar to an observation of Greenleaf, Pramanik, and Tang \cite{gpt2007} concerning oscillatory integral operators.
\end{abstract}

The purpose of this paper is to establish endpoint $L^p$-$L^q$ and Sobolev inequalities for a broad class of highly degenerate Radon-like averaging operators.  The literature relating to this goal is both broad and deep, including but not limited to the works of Bak, Oberlin, and Seeger \cite{bos2002}; Cuccagna \cite{cuccagna1996}; Greenblatt \cite{greenblatt2005}; Greenleaf and Seeger \cite{gs1994}, \cite{gs1998}; Lee \cite{lee2003}, \cite{lee2004}; Phong and Stein \cite{ps1997}; Phong, Stein, and Sturm \cite{pss2001}; Pramanik and Yang \cite{py2004}; Rychkov \cite{rychkov2001}; Seeger \cite{seeger1998}; and Tao and Wright \cite{tw2003}.  This literature provides a comprehensive theory of Radon transforms in the plane (optimal $L^p$-$L^q$ and Sobolev bounds were established by Seeger \cite{seeger1998} and others).  Tao and Wright \cite{tw2003} have also established sharp (up to $\epsilon$ loss) $L^p$-$L^q$ inequalities for completely general averaging operators over curves in any dimension.

In the remaining cases, though, little has been proved regarding optimal inequalities for Radon-like operators.  Among the reasons for this is that the rotational curvature (in the sense of Phong and Stein \cite{ps1986I}, \cite{ps1986II}) is essentially controlled by a scalar quantity for averaging operators in the plane, but is governed in higer dimensions (and higher codimension) by a matrix condition which is increasingly difficult to deal with using standard tools.  While it is generally impossible for rotational curvature to be nonvanishing in this case, the corresponding matrix can be expected to have nontrivial rank.  Under this assumption, works along the lines of Cuccagna \cite{cuccagna1996} and Greenleaf, Pramanik, and Tang \cite{gpt2007} have been able to use this weaker information as a replacement for nonvanishing rotational curvature.  In particular, Greenleaf, Pramanik, and Tang showed that optimal $L^2$-decay inequalities for ``generic'' oscillatory integral operators can be established in the highly degenerate case with only the knowledge that the corresponding matrix quantity has rank one or higher at every point away from the origin.  The purpose of this paper, then, is to explore and extend this phenomenon as it can be applied to the setting of Radon-like operators.

Fix positive integers $n'$ and $n''$, and let $S$ be a smooth mapping into $\R^{n''}$ which is defined on a neighborhood of the origin in $\R^{n'} \times \R^{n''} \times \R^{n'}$.  The purpose of this paper is to prove a range of sharp $L^p-L^q$ and Sobolev inequalities for the Radon-like operator defined by
\begin{equation}
T f(x',x'') := \int f(y',x'' + S(x',x'',y')) \psi(x',x'',y') dy', \label{theop}
\end{equation}
where $x',y' \in \R^{n'}$ and $x'' \in \R^{n''}$ ($n'$ represents the dimension of the manifolds over which $f$ is averaged, and $n''$ represents the codimension).  When no confusion arises, the variable $x$ will stand for the pair $(x',x'')$, and $n$ will refer to the sum $n' + n''$.

The assumption to be made on $S$ is that it exhibits a sort of approximate homogeneity (aka semiquasihomogeneity).  The notation to be used to describe this scaling will be as follows:  given any multiindex $\gamma := (\gamma_1,\ldots,\gamma_m)$ of length $m$, any $z := (z_1,\ldots,z_m) \in \R^m$, and any integer $j$, let $2^{j \gamma} z := (2^{j \gamma_1} z_1, \ldots 2^{j \gamma_m} z_m)$.  The entries of a multiindex will always be integers, but they will be allowed to be negative in situations where negative entries make sense.  The order of the multiindex $\gamma$ will be denoted $|\gamma|$, is the sum of the entries, i.e., $\gamma_1 + \cdots + \gamma_m$, and may be negative in some cases.

As for the mapping $S$, it will be assumed that there exist multiindices $\alpha'$ and $\beta'$ of length $n'$ and $\alpha''$ and $\beta''$ of length $n''$ such that the limit of
\begin{equation} \lim_{j \rightarrow \infty} 2^{j \beta''} S( 2^{-j \alpha'} x', 2^{-j \alpha''} x'', 2^{-j \beta'} y') =: S^{P}(x',x'',y')\label{homo1}
\end{equation}
as $j \rightarrow \infty$ exists and is a smooth function of $x'$, $x''$, and $y'$ which does not vanish identically (note that, given a smooth mapping $S$, there is always at least one choice of multiindices so that this condition holds).  Furthermore, it will be assumed that $\beta''_i > \alpha''_i$ for $i=1,\ldots,n''$.  The assumption on $\alpha''$ and $\beta''$ will together with \eqref{homo1} be referred to as the homogeneity conditions. 

As with the variable $x$, the multiindices $\alpha$ and $\beta$ of length $n$ will represent $(\alpha',\alpha'')$ and $(\beta',\beta'')$ respectively.  Although the mapping $S$ exhibits a weak sort of homogeneity, the second of the homogeneity conditions guarantees, in fact, that the averaging operator \eqref{theop} is {\it not} homogeneous.  For this reason, it turns out that there is more than one family of dilations that come into play in the study of \eqref{theop}.  To simplify the proofs somewhat, it will also be convenient to define $\tilde \alpha$ to represent $(\alpha',\beta'')$.

The main nondegeneracy condition to be used is stated as follows:  for each pair $(x,y')$ in the support of the cutoff $\psi$ in \eqref{theop} and each $\eta'' \in \R^{n''} \setminus \{ 0 \}$, consider the $n' \times n'$ mixed Hessian matrix $H^P$ whose $(i,j)$-entry is given by
\begin{equation} H^P_{ij} (x',x'',y',\eta'') := \frac{\partial^2}{\partial {x_i'} y_j'} \left( \eta'' \cdot S^P(x',x'',y') \right). \label{hessian}
\end{equation}
Throughout the paper, it will be assumed that there is a positive integer $r > 0$ such that, at each point $(x',x'',y') \neq (0,0,0)$ and for each $\eta'' \neq 0$, the matrix $H^P(x',x'',y',\eta'')$ has rank at least $r$.  This condition is very closely related to the condition of nonvanishing rotational curvature of Phong and Stein \cite{ps1986I}, \cite{ps1986II}; however, even when $r$ is maximal, the operators \eqref{theop} can and generally do have vanishing rotational curvature at the origin.  When $r < n'$ the rotational curvature may actually vanish at every point.

\begin{theorem}
Suppose that the operator \eqref{theop} satisfies the homogeneity conditions and that the mixed Hessian \eqref{hessian} has rank at least $r$ at every $(x,y',\eta'') \neq (0,0,0)$.  If the support of $\psi$ is sufficiently near the origin and $\frac{r}{n''} > \frac{|\alpha'|+|\beta'|}{|\beta''|}$ then $T$ maps $L^p(\R^n)$ to $L^q(\R^n)$ provided that the following inequalities are satisfied:
\begin{equation} 
\frac{|\beta'| + |\beta''|}{p} - \frac{|\alpha'| + |\beta''|}{q} < |\beta'|, \label{condition1}
\end{equation}
\begin{equation}
\left| \frac{1}{p} + \frac{1}{q} - 1 \right| < 1 - \frac{2 n'' + r}{r} \left( \frac{1}{p} - \frac{1}{q} \right). \label{condition2}
\end{equation}
Additionally, $T$ maps $L^p$ to $L^q$ if either one of the inequalities \eqref{condition1} or \eqref{condition2} are replaced with equality.  If both inequalities are replaced with equality, then $T$ is of restricted weak-type $(p,q)$.  The Riesz diagram corresponding to these estimates is shown in figure \ref{riesz}. \label{lplqthm}
\end{theorem}

\begin{figure}
\centering
\begin{pspicture}(-0.5,-0.5)(6.5,6)
\pspolygon[linewidth=1pt](0,6)(0,0)(6,0)(6,6)
\pspolygon*[linewidth=0.5pt,linecolor=lightgray](0,0)(3.8888,0.7777)(5.6666,4.3333)(6,6)(0,6)
\pspolygon[linewidth=0.5pt](0,0)(3.8888,0.7777)(5.6666,4.3333)(6,6)(0,6)
\psline[linewidth=0.5pt,linestyle=dashed](3.5,0)(3.8888,0.7777)(5,1)(5.6666,4.3333)(6,5)
\put(-0.4,2.9){$\displaystyle \frac{1}{q}$}
\put(2.8,-0.4){$\displaystyle 1/p$}
\pscircle(3.8888,0.7777){0.1}
\pscircle(5.6666,4.3333){0.1}
\end{pspicture}
\caption{Riesz diagram corresponding to the estimates proved for $T$ (the shaded area) in the case when $\frac{r}{n''} > \frac{|\alpha'| + |\beta'|}{|\beta''|}$.  Restricted weak-type inequalities are obtained at the nontrivial vertices (marked by circles).}
\label{riesz}
\end{figure}
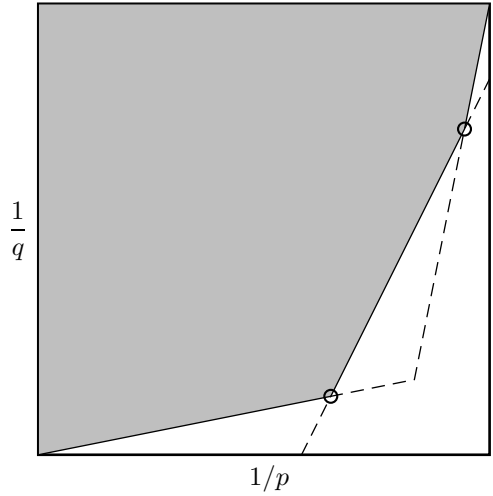

In the event that $\frac{r}{n''} < \frac{|\alpha'| + |\beta'|}{|\beta''|}$, then the condition \eqref{condition2} excludes the possibility of equality in condition \eqref{condition1}.  Both \eqref{condition1} and \eqref{condition2} are ``optimal,'' but in varying senses.  In the case of the former, any operator satisfying the homogeneity condition is unbounded from $L^p$ to $L^q$ if $\frac{|\beta'| + |\beta''|}{p} - \frac{|\alpha'| + |\beta''|}{q} > |\beta'|$.  For the latter, there exists an operator satisfying the rank condition for which \eqref{condition2} cannot hold when $(p,q)$ satisfy the reverse inequality.

\begin{theorem}
Suppose that $T$ satisfies the rank and homogeneity conditions and $\frac{r}{n''} > \frac{|\alpha'| + |\beta'|}{|\beta''|}$ (and the support of $\psi$ is sufficiently near the origin).  Then 
the operator $T$ maps the space $L^p(\R^n)$ to the Sobolev space $L^p_s(\R^n)$ ($s \geq 0$) provided that the following two conditions are satisfied: \label{sobolevthm}
\begin{equation}
 s \max\{ \beta''_1,\ldots,\beta''_{n''}\} \leq \frac{|\alpha'|}{p} + |\beta'| \left(1- \frac{1}{p}\right), \label{condition3}
\end{equation}
\begin{equation}
\frac{s}{r} < \frac{1}{2} - \left|\frac{1}{2} - \frac{1}{p} \right|. \label{condition4}
\end{equation}
\end{theorem}

Just like the constraint \eqref{condition1}, the inequality \eqref{condition3} is necessarily satisfied by any operator satisfying the hypotheses of theorem \ref{sobolevthm}.  Another interesting feature of theorems \ref{lplqthm} and \ref{sobolevthm} is that the homogeneity condition and the rank condition on the Hessian are decoupled, in the sense that each of the constraints \eqref{condition1}--\eqref{condition4} depends (quantitatively, at least) on only one of the two assumptions made of $T$.  A consequence of this is that when $|\beta''|$ is large, the $L^p$-$L^q$ boundedness of $T$ near the line of duality $\frac{1}{p} + \frac{1}{q} = 1$ as well as the $L^p$-$L^p_s$ boundedness for $p$ near $2$ are almost completely insensitive to the condition on the rank of the Hessian.  This pheonomenon was observed by Greenleaf, Pramanik, and Tang \cite{gpt2007} in the context of oscillatory integral operators (this is the ``low-hanging fruit'').

It is of particular interest, then, to make a statement quantifying the strength of the rank assumption on the mixed Hessian \eqref{hessian}.  For some particular combinations of $\alpha',\alpha'',\beta'$, and $\beta''$, there may not, in fact, be any operators satisfying the homogeneity condition because $S$ is assumed to be smooth.  For this reason, it will not be possible to make a meaningful statment valid for every possible combination of multiindices.  To rectify, the multiindices $\alpha',\beta'$ and $\alpha''$ will be considered fixed, and a ``positive fraction'' of the choices of $\beta''$ will be examined.  There are a variety of ways to formulate this concept;  here a set of multiindices $E$ of length $n''$ will be said to have lower density $\epsilon$ provided that
\[ \liminf_{N \rightarrow \infty} \frac{\# \set{\beta'' \in E}{\beta''_i \leq N \ \forall i=1,\ldots,n''}}{N^{n''}} \geq \epsilon. \]
Let $\Lambda_{\alpha,\beta}$ be the space of all $n''$-tuples of real polynomials $(p_1,\ldots,p_{n''})$ in the variables $x',y',$ and $x''$ (for $x',y' \in \R^{n'}$ and $x'' \in R^{n''}$) such that
\[ p_l(2^{j \alpha'} x', 2^{j \alpha''} x'', 2^{j \beta'} y') = 2^{\beta''_l} p_l(x',x'',y') \]
for each integer $j$ and $l=1,\ldots,n''$; suppose further that $\Lambda^{\alpha,\beta}$ is given the topology of a real, finite-dimensional vector space.  Each element $(p_1,\ldots,p_{n'})$ naturally induces an operator of the form \eqref{theop} which satisfies the homogeneity condition.  The strength of the condition \eqref{hessian} can now be quantified as follows:
\begin{theorem}
Fix $\alpha',\alpha''$, and $\beta'$.
Let $K_1$ be the least common multiple of the entries of $\alpha',\alpha''$ and $\beta'$; let $K_2$ be the number of distinct values (modulo $K_1$) taken by the sum $\alpha_i' + \beta_j'$ for $i,j=1,\ldots,n'$.  Then for any $\beta''$ in some set of lower density $K_1^{-n''}$, the operators \eqref{theop} corresponding to the polynomials $\Lambda_{\alpha,\beta}$ generically satisfy the rank condition provided $$r < n' - \sqrt{(1-K_2^{-1})(n')^2 + 2n}.$$
\end{theorem}

In the context of averages over hypersurfaces with isotropic homogeneity (taking the entries of $\alpha'$, $\alpha''$, and $\beta'$ to equal one, corresponding to the case $n_X = n_Z$ in the work of Greenleaf, Pramanik, and Tang \cite{gpt2007}), a generic mixed Hessian \eqref{hessian} has everywhere (exept the origin) rank at least $n' - 1 - \sqrt{2n'+2}$, and the hypotheses of the theorems \ref{lplqthm} and \ref{sobolevthm} are satisfied for {\it any} choice of $\beta'' \geq 3$ when $n' > 25$.  On the opposite extreme, a rank one condition holds provided that $n'' < \frac{n'(n'-4)}{2}$ (an extremely large codimension, similar to those encountered by Cuccagna \cite{cuccagna1996} and well beyond the range of nonvanishing rotational curvature) and theorems \ref{lplqthm} and \ref{sobolevthm} hold with $r=1$ for all multiindices $\beta''$ satisfying $|\beta''| > (n')^2 (n'-4)$.


\section{Preliminaries}
To begin, a few comments about homogeneity are necessary.  Given multiindices $\alpha', \alpha''$, and $\beta'$, a smooth function $f$ defined on a neighborhood of the origin in $\R^{2n' + n''}$ (as a function of  $x',x''$, and $y'$) will be called nearly homogeneous of degree $l \geq 0$ if 
\begin{equation}
\lim_{j \rightarrow \infty} 2^{lj} f(2^{-j \alpha} x, 2^{-j \beta'} y') =: f^{P}(x',x'',y') \label{principal}
\end{equation}
 exists as $j \rightarrow \infty$ for every $x', x''$, and $y'$ and is nonzero at some point (the limit function $f^{P}$ will be called the principal part of $f$).  In fact, given any smooth function $f$ not vanishing to infinite order at the origin, there is a unique nonnegative integer $l$ such that $f$ is nearly homogeneous of degree $l$, the limit \eqref{principal} must be uniform on compact sets, and the prinicpal part must be a polynomial.  Furthermore, for any multiindex $\gamma$ of length $n$ and any multiindex $\delta$ of length $n'$,
\begin{equation} 
\lim_{j \rightarrow \infty} 2^{j l - j \alpha \cdot \gamma - j \beta' \cdot \delta } \frac{\partial^{|\gamma|+|\delta|} f}{(\partial x)^\gamma(\partial y')^\delta} (2^{-j \alpha}x,2^{-j \beta'} y') = \frac{\partial^{|\gamma| + |\delta|} f^{P}}{(\partial x)^\gamma (\partial y')^\delta}(x,y') \label{limits}
\end{equation}
(where $\alpha \cdot \gamma = \sum_{i=1}^n \alpha_i \gamma_i$ and likewise for $\delta \cdot \beta'$) with uniform convergence on compact sets.  These assertions all follow directly from Taylor's theorem with remainder by regrouping terms according to homogeneity degree (with respect to $\alpha', \alpha''$, and $\beta'$); the proofs are very classical and will not be given here.

\subsection{Remarks on the dual operator $T^*$}
The next item to be explored is the nature of the operator $T^*$ which is dual to \eqref{theop}.  For fixed $x'$ and $y'$, let $\Phi_{x',y'}(x'') := x'' + S(x',x'',y')$.  To express the operator $T^*$ as an integral operator, it is necessary to invert the mapping $\Phi_{x',y'}$.  To that end, consider the derivative of $S_l$ with respect to $x_k''$.  The function $S_l$ is nearly homogeneous of degree $\beta''_l$ by assumption; therefore \eqref{limits} guarantees that the derivative $\partial_{x''_k} S_l(x,y')$ vanishes at the origin whenever $\beta''_l > \alpha''_k$.
  Since $\beta_i'' > \alpha_i''$ for each $i$, it follows that the Jacobian matrix of $\Phi_{0,0}(x'')$ at $x''=0$ is upper triangluar with ones along the diagonal (after a suitable permutation of the rows and columns).  As a result, the inverse function theorem guarantees the existence of a smooth inverse to $\Phi_{x',y'}$ near $x''=0$ for all sufficiently small $x'$ and $y'$.  It follows that the operator dual operator $T^*$ may be written as
\[ T^* g(y',y'') = \int g(x',y''-S(x',\Phi_{x',y'}^{-1}(y''),y')) \tilde \psi(x',y) dx' \]
for a new cutoff $\tilde \psi$ equal to the old cutoff $\psi$ divided by the absolute value of the Jacobian determinant of $\Phi_{x',y'}$.

The next step is to compute the principal part of the dual mapping $S^*$ defined by $S^*(y',y'',x') := -S(x',\Phi_{x',y'}^{-1}(y''),y')$.
To do so, consider yet another important consequence of the assumption $\beta''_i > \alpha''_i$: \begin{equation}\lim_{j \rightarrow \infty} 2^{j \alpha''} \Phi_{2^{-j \alpha'} x', 2^{-j \beta'} y'}(2^{-j \alpha''} x'') = x''\label{limit2} \end{equation} with uniform convergence on compact sets.  Furthermore, for any $R > 0$, the inverse function theorem provides a uniform constant $C_R$ such that 
\begin{equation} |x''| \leq C_R | 2^{j \alpha''} \Phi_{2^{-j \alpha'} x', 2^{-j \beta'} y'}(2^{-j \alpha''} x'')| \label{limit3}
\end{equation}
uniform in $j$, valid for all $x''$ such that the right-hand side is itself bounded by $R C_R$.  It therefore must be the case that
$$\lim_{j \rightarrow \infty} 2^{j \alpha''} \Phi_{2^{-j \alpha'} x',2^{-j \beta'} y'}^{-1} (2^{-j \alpha''} y'') = y''$$ (since \eqref{limit3} guarantees that the sequence on the left-hand side is bounded, and the uniformity of \eqref{limit2} shows that any convergent subsequence must have limit $y''$).  Consequently, if $S$ satisfies the homogeneity condition and has principal part $S^{P}(x',x'',y')$, then $S^*(x',y'',y')$ satsifies the homogeneity condition with $x'$ scaled by $\alpha'$, $y'$ by $\beta'$, and $y''$ by $\alpha''$ with principal part $-S^{P}(x',y'',y')$.  Thus if \eqref{theop} has a mixed Hessian \eqref{hessian} with rank at least $r$ near the origin, then so does $T^*$.  This fact will be used later to simplify the proof of theorem \ref{sobolevthm}.

\subsection{Main tools}

Now comes the time to prove the main tools which power the arguments necessary for theorems \ref{lplqthm} and \ref{sobolevthm}.  To simplify matters, fix, once and for all, a smooth function $\varphi_0$ on the real line which is supported on $[-2,2]$, identically one on $[-1,1]$, and monotone on $[0,\infty)$ and $(-\infty,0]$.

The first order of business is the integration-by-parts lemma.  The key idea of the method of stationary phase is that the main contributions to an oscillatory integral occur where the gradient of the phase is ``small.''  While there is an intrinsic way of stating that the gradient of the phase vanishes, there is (unfortunately) no coordinate-independent way of quantifying ``smallness.'' The answer, then, is to be explicit about the coordinates being used, and to change those coordinates whenever it is necessary and proper to do so. 

This changing of coordinate systems is captured here by what will be called scales.  More precisely:  a scale $\scaled$ on $\R^d$ will be any multiindex of length $d$ with entries in $\Z$.  A vector $v \in \R^d$ will have length relative to $\scaled$ given by
\[ |v|_\scaled := \left( \sum_{i=1}^d 2^{2 \scaled_i} |v_i|^2 \right)^\frac{1}{2} \]
(the term ``scale'' was chosen because $\scaled$ implicitly induces a rescaling of the standard coordinate system via this formula).  Likewise, the derivative $\partial^\gamma_\scaled$ (for some standard multiindex $\gamma$) is meant to represent the derivative
\[ 2^{\sum_{i=1}^d \gamma_i \scaled_i} \left( \frac{\partial}{\partial t_1} \right)^{\gamma_1} \cdots \left( \frac{\partial}{\partial t_d} \right)^{\gamma_d} \]
(where the standard coordinates are here labelled $t_1,\ldots,t_d$).  In this notation, the integration-by-parts lemma is stated as follows:
\begin{lemma}
Let $\Phi$ be any real-valued, $C^\infty$ function defined on some open subset of $\R^{d}$, and let $\varphi$ be a $C^\infty$ function compactly supported in the domain of $\Phi$.  Then for any positive integer $N$, there exists a constant $C_N$ such that \label{ibplemma}
\begin{equation} \left| \int e^{i \Phi(t)} \varphi(t) dt \right| \leq C_N \! \int \frac{\sum_{|\gamma|=0}^N | \partial_\scaled^\gamma \varphi(t)|}{(1 +  \epsilon |\nabla \Phi(t)|_\scaled)^N} \left|   1 + \frac{\epsilon^2 \sum_{|\gamma|=2}^{N+1} | \partial_\scaled^\gamma \Phi(t)|}{1 + \epsilon |\nabla \Phi(t)|_\scaled} \right|^N \! dt, \label{ibp}
\end{equation}
where $\scaled$ is any scale and $0 < \epsilon \leq 1$.
\end{lemma}
\begin{proof}
Consider the following integral on $\R^{d+1}$:
\[ I_\alpha := \int e^{i (- 2 \pi \alpha t_0 + \Phi(t))} \varphi_0(t_0) \varphi(t) dt_0 dt. \]
There is at least one value of $\alpha \in (0,1)$ depending only on $\varphi_0$, there exists a constant $C_\alpha \neq 0$ such that $C_\alpha^{-1} I_\alpha$ is precisely the value of the integral to be computed.  Let such an $\alpha$ be fixed once and for all.  Let $\tilde t := (t_0,t_1,\ldots, t_d)$, and likewise let $\tilde \Phi(\tilde t)$ and $\tilde \varphi(\tilde t)$ represent the phase and amplitude, respectively, appearing in the integral defining $I_\alpha$.

Let $k$ be any nonnegative integer, let $l$ be an integer such that $2^{-l} < 2 \pi \alpha \leq 2^{-l+1}$,  and let $\tilde \scaled$ be the scale on $\R^{d+1}$ given by $(l,\scaled_1-k,\ldots,\scaled_n-k)$.  Now consider the following differential operator on $\R^{d+1}$:
\[ L_{\tilde \scaled} f(\tilde t) := \frac{ \nabla_{\tilde \scaled} \tilde \Phi(\tilde t) \cdot \nabla_{\tilde \scaled} f(\tilde t)}{i |\nabla \tilde \Phi( \tilde t)|_{\tilde \scaled}^2}. \]
Since $\alpha \neq 0$, the operator $L_{\tilde \scaled}$ is well-defined because the denominator is nonzero.  The standard integration-by-parts argument dictates that
\begin{equation*}
I_\alpha =  \int \left((L_{\tilde \scaled})^N e^{i \tilde \Phi(\tilde t)} \right) \tilde \varphi(\tilde t) d \tilde t = \int e^{i \tilde \Phi(\tilde t)} \left( (L_{\tilde \scaled}^t)^N \tilde \varphi(\tilde t) \right) d \tilde t 
\end{equation*}
for each integer $N \geq 0$, where $L_{\tilde \scaled}^t$ is the adjoint of $L_{\tilde \scaled}$.  Now an elementary induction on the Leibnitz rule gives that, for each $N$, there is a constant $C_{N}$ depending on $N$ (and the dimension $d$) such that
\begin{equation}|(L_{\tilde \scaled}^t)^N \tilde \varphi(\tilde t)| \leq \frac{C_{N}}{|\nabla \tilde \Phi(\tilde t)|_{\tilde \scaled}^N} \left( \frac{\sum_{|\gamma|=1}^{N+1} |\partial_{\tilde \scaled}^\gamma \tilde \Phi(\tilde t)|}{|\nabla \tilde \Phi(\tilde t)|_{\tilde \scaled}} \right)^N \sum_{|\gamma'|=0}^N |\partial_{\tilde \scaled}^{\gamma'} \tilde \varphi (\tilde t)|. \label{ibp1}
\end{equation}
At this point, several simplifications are in order.  First, observe that $|\nabla \tilde \Phi(\tilde t)|_{\tilde \scaled}^2 = 2^{2l} 4 \pi^2 \alpha^2 + 2^{-2k} |\nabla \Phi(t)|_\scaled^2 > 1 + 2^{-2k} |\nabla \Phi(t)|_\scaled^2$.  Next, 
\[ \sum_{|\gamma|=2}^N | \partial_{\tilde \scaled}^\gamma \tilde \Phi(\tilde t)| = \sum_{|\gamma|=2}^N 2^{-k |\gamma|} | \partial_{\scaled}^\gamma \Phi(t)| \leq 2^{-2k} \sum_{|\gamma|=2}^N | \partial_{\scaled}^\gamma \Phi(t)|\]
(where $\gamma$ represents a multiindex of length $d+1$ on the left-hand side and length $d$ in the middle and on the right) because $\tilde \Phi(\tilde t)$ differs from $\Phi(t)$ by a linear term.  Finally,
\[ \sum_{|\gamma'|=0}^N |\partial_{\tilde \scaled}^{\gamma'} \tilde \varphi (\tilde t)| \leq C_N' \chi_{[-2,2]}(t_0) \sum_{|\gamma'|=0}^N |\partial_{\scaled}^{\gamma'} \varphi ( t)| \]
again by the Leibnitz rule and the compact support of $\varphi_0$.  Combining these three observations with the inequality \eqref{ibp1} and performing the (trivial) integral over $t_0$ first gives \eqref{ibp} if $k$ is chosen so that $2^{-k} \leq \epsilon \leq 2^{-k+1}$.
\end{proof}
The second idea to be used repeatedly throughout all that follows is contained in the proposition below.  In simplest terms, the result is that the integral of certain simple ratios (which appear will appear frequently) can be estimated by removing appropriate terms from the denominator and multiplying by an appropriate factor of two coming from the scale:
\begin{proposition}
For any multiindex $\gamma$ and any positive integer $N$ sufficiently large (depending only on $\gamma$ and the dimension), there is a constant $C_{N,\gamma}$ such that
\begin{equation} \int \frac{|t^\gamma|}{(|\tau| + |t|_\scaled)^N} dt\leq C_{N,\gamma} \frac{2^{-|\scaled| - \gamma \cdot \scaled}}{|\tau|^{N-d - |\gamma|}} \label{inteqn}
\end{equation}
for any scale $\scaled$ and any real $\tau$.  \label{intlemma}
\end{proposition}
\begin{proof}
The inequality \eqref{inteqn} follows immediately from a change of variables.  Changing $t_i \mapsto |\tau| 2^{-\scaled_i}$ for $i=1,\ldots,d$, the desired integral is equal to
\[ 2^{-|\scaled| - \gamma \cdot \scaled} |\tau|^{|\gamma| + d - N} \int \frac{|t^\gamma|}{(1 + |t|)^N} dt, \]
and this new integral is clearly finite when $N > |\gamma| + d$.
\end{proof}

\subsection{Fractional differentiation}

An essential component of theorem \ref{sobolevthm}.  Here it will be useful to develop nonisotropic versions of the standard Bessel potentials (found, for example in Stein \cite{steinsi}).  Since the operator \eqref{theop} is not actually homogeneous, however, there will be more than one natural choice of scaling to use in defining the nonisotropic Bessel potentials; not only that, it will be necessary to make certain estimates of these Bessel potentials using {\it conflicting} families of dilations.  For this reason, it is worthwile to proceed in nearly complete generality and work with a large family of potentials.

Recalling the fixed function $\varphi_0$ on the real line, let $\varphi_\Pi(\xi) := \prod_{i=1}^n \varphi_0(\xi_i)$ (clearly $\xi \in \R^n$).  Given any multiindex $\gamma$ (with strictly positive entries) and any complex number $s$ satisfying $\real{s} \geq 0$, consider the tempered distribution $J_{\gamma}^s$ whose Fourier transform is given by
\begin{equation} 
(J_{\gamma}^s)^\wedge(\xi) := \varphi_\Pi(\xi) + \sum_{j=1}^\infty 2^{s j} \left[ \varphi_\Pi(2^{- j \gamma} \xi ) - \varphi_\Pi(2^{-(j-1) \gamma} \xi) \right].  \label{fracdifdef}
\end{equation}
Note that when $s$ is real, $(J_\gamma^s)^\wedge(\xi)$ is nonnegative for all $\xi$ by the monotonicity conditions on $\varphi_0$.  A function $f$ on $\R^n$ will be said to belong to the space $L^{p}_{s,\gamma}(\R^n)$ provided that $||J_\gamma^s \star f||_p < \infty$.  When $s$ is real, $\gamma = \one := (1,\ldots,1)$ and $1 < p < \infty$, the Calder\'on-Zygmund theory of singular integrals guarantees that the space $L^p_{s,\gamma}(\R^n)$ is the usual Sobolev space.  More generally, the space $L^p_{s,\gamma}(\R^n)$ can be thought of as the space of functions which are differentiable to order $s /\gamma_i$ in the $i$-th coordinate direction.  It also follows that $\partial^l f \in L^p$, $1 < p < \infty$, provided that $l \cdot \gamma \leq s$.

As there are various scalings to be exploited in the proofs to follow, it is necessary to record the behavior of the distribution $J_\gamma^s$ when it is restricted to a box which has a potentially different scaling $\delta$ (that is, a box with side lengths approximately $2^{-\delta_i}$ for $i=1,\ldots,n$).  For this reason, consider the distribution obtained from multiplying $J_\gamma^s$ by the Schwartz function $\varphi_\Pi(2^\delta x)$.  The resulting distribution will be called $ J_{\gamma}^s |_{\delta}$; its Fourier transform is given by the convolution
\begin{equation} \left(  \left.  J_{\gamma}^s  \right|_{\delta} \right)^\wedge (\xi)  = 2^{- |\delta|} \int \hat \varphi_\Pi(2^{- \delta} (\xi - \eta)) (J_\gamma^s)^\wedge(\eta)  d \eta. \label{fracdifres}
\end{equation}
Now $\varphi_\Pi(2^{- j \gamma} \xi ) - \varphi_\Pi(2^{-(j-1) \gamma} \xi) = 0$ when $|2^{-j \gamma_i} \xi_i| \geq 1$ for some any value of $i$.  It follows that on the support of $\varphi_\Pi(2^{- j \gamma} \xi ) - \varphi_\Pi(2^{-(j-1) \gamma} \xi)$, $2^{\real{s} j} \leq (\frac{1}{2} |\xi_i|)^{\real{s}/\gamma_i}$.  Hence it follows that
\[ |(J_\gamma^s)^\wedge(\xi)| \leq 1 + \sum_{i=1}^n  2^{-\frac{\real{s}}{\gamma_i}} |\xi_i|^\frac{\real{s}}{\gamma_i}. \]
Inserting this inequality into \eqref{fracdifres} gives that
\[ \left|\left( \left. J_{\gamma}^s \right|_\delta \right)^\wedge(\xi) \right| \leq 2^{- |\delta|} \int |\hat \varphi_\Pi(2^{- \delta}(\xi - \eta))|  \left( 1 + \sum_{i=1}^n 2^{-\frac{\real{s}}{\gamma_i}}|\eta_i|^\frac{\real{s}}{\gamma_i} \right) d \eta. \]
Now when $\real{s} \geq 0$, $2^{-\real{s}/\gamma_i} |\eta_i|^{\real{s}/\gamma_i} \leq |\xi_i|^{\real{s}/\gamma_i} + |\eta_i - \xi_i|^{\real{s}/\gamma_i}$;  the result is that there exists a constant $C$ independent of $\delta$ and $\imaginary{s}$ such that
\begin{equation*}
 \left| \left(\left. J_{\gamma}^s \right|_{\delta} \right)^\wedge (\xi) \right| \leq C  \left( 2^{\real{s} \frac{\delta}{\gamma}} + \sum_{i=1}^n |\xi_i|^{\frac{\real{s}}{\gamma_i}} \right).  
\end{equation*}
The same procedure yields the more general family of inequalities
\begin{equation}
 \left| \partial_{\delta}^{l}  \left(\left. J_{\gamma}^s \right|_{\delta} \right)^\wedge (\xi) \right| \leq C \left( 2^{\real{s} \frac{\delta}{\gamma}} + \sum_{i=1}^n |\xi_i|^{\frac{\real{s}}{\gamma_i}} \right)  \label{diff1}
\end{equation}
where, again, the constant does not depend on $\delta$ or $\imaginary{s}$.  This inequality will be indispensible in applying the integration-by-parts lemma in the presence of a fractional differentiation which is not of the same sort of scaling as the rest of the integral.

The standard arguments appearing in the theory of regular homogeneous distributions guarantee that $J_\gamma^s -  J_\gamma^s|_\delta$ is a $C^\infty$ function which is, in fact, of rapid decay.  Let $\Delta_j(x)$ be the inverse Fourier transform of the difference $\varphi_{\Pi}(2^{-j \gamma} \xi)$.  The usual integration-by-parts arguments require that, for each positive integer $N$, there exists a constant $C$ such that $\Delta_0(x)| \leq C_{N,l} (1 + |x|^N)^{-1}$.  Rescaling, it follows that 
\[|\Delta_j(x)| \leq \frac{C_N 2^{j |\gamma|}}{1 + |x|_{j \gamma}^N}  \]
for the same constant $C_N$.  Let $\frac{\delta}{\gamma}$ be defined for any two multiindices $\delta$ and $\gamma$ of the same length to equal the maximum value of $\frac{\delta_i}{\gamma_i}$ as $i$ ranges over all entries.  Now for any multiindices $\delta$ and $\gamma$,
\begin{align*}
|x|_\delta & := \left( \sum_{i=1}^n |2^{\delta_i} x_i|^2 \right)^\frac{1}{2} = \left( \sum_{i=1}^n \left|2^{\left(\frac{\delta_i}{\gamma_i}-j  \right) \gamma_i} 2^{j \gamma_i} x_i\right|^2 \right)^\frac{1}{2} \\
& \leq 2^{ \left( \frac{\delta}{\gamma} - j \right) \min_i \gamma_i } |x|_{j \gamma}
\end{align*}
provided $j \geq \frac{\delta}{\gamma}$.  Therefore, taking the inverse Fourier transform of the right-hand side \eqref{fracdifdef} and integrating over the set of $x$'s where $|x|_{\delta} \geq \frac{1}{2}$ gives that
\[ \left|\left| J_\gamma^s - \left.  J_\gamma^s \right|_\delta \right|\right|_1 \leq C_{N} \left[ \sum_{0 \leq j \leq \frac{\delta}{\gamma}} 2^{\real{s}j + j \gamma \cdot l} + \sum_{j > \frac{\delta}{\gamma}} 2^{\real{s} j + j \gamma \cdot l + N (\frac{\delta}{\gamma} - j) \min_i \gamma_i} \right]. \]
Choosing $N$ sufficiently large guarantees that
\begin{equation} \left|\left| \left( J_\gamma^s - \left. J_\gamma^s \right|_\delta \right) \star f \right|\right|_p \leq C_{\gamma,\real{s}} 2^{\real{s} \frac{\delta}{\gamma}} ||f||_p  \label{diff2}
\end{equation}
for all $1 \leq p \leq \infty$, uniform in $\imaginary{s}$ and $\delta$.  When no confusion will arise, the convolution operators corresponding to convolution with $J_\gamma^s$ and $J_\gamma^s|_\delta$ will simply be written as $J_\gamma^s$ and $J_\gamma^s|_\delta$ (i.e., the star will be supressed).

\subsection{Main decomposition}
The time has now come to describe the decomposition of the operator \eqref{theop} which will be used to prove theorems \ref{lplqthm} and \ref{sobolevthm}.  The first step, as is easily imagined, is to decompose the support of the operator \eqref{theop} away from the origin $(x',x'',y') = (0,0,0)$ in a way that is consistent with the scalings of the homogeneity condition.  Given an amplitude $\psi$ supported near the origin, fix some smooth function $\varphi$ on $\R^{n'} \times \R^{n''} \times \R^{n'}$ which is identically one on the support of $\psi$ and is itself compactly supported.  Now let
\[ \psi_j (x,y') := \psi(x,y') (\varphi(2^{j \alpha} x, 2^{j \beta'} y') - \varphi(2^{(j+1) \alpha} x, 2^{(j+1) \beta'} y')) \]
and consider the following two families of operators:
\begin{align*} T_j f(x) & := \int f(y',x'' + S(x,y')) \psi_j(x,y') dy', \\
U_j f(x) & := \int f(y',x'' + S(x,y')) \psi(x,y') \varphi(2^{j \alpha} x, 2^{j \beta'}y') dy' = \sum_{l=j}^\infty T_j f(x).
\end{align*}
Clearly $T = \sum_{j=0}^\infty T_j$ suitably defined.  For example, if $f(y',y'')$ is a Schwartz function on $\R^n$ whose support is at a nonzero distance from the hyperplane $y' = 0$, then $Tf = \sum_{j=0}^\infty T_j f$ with convergence in the Schwartz space topology (and, in fact, only finitely many terms of the sum are nonzero).  This is because the supports of $T_j f$ and $U_j f$ are contained in a box of side lengths comparable to $2^{-j \alpha_i}$ for $i=1,\ldots,n$, and the supports in $y'$ of the cutofffs for both operators are similarly restricted to a box of sides $2^{-j \beta''_i}$ for $i=1,\ldots,n'$.

The operators $T_j$ will be further decomposed (according to a new family of dilations which is potentially in conflict with the one already used).  To that end, choose $\tilde \varphi$ to be a smooth function of compact support on $\R^{n''}$ which is supported in the Euclidean ball of radius $1$ and is identically one on the ball of radius $\frac{1}{2}$.  Now for any nonnegative integers $j,k$, let
\begin{align*}
(P_{jk} f)^\wedge (\xi',\xi'') & :=  \left[  \tilde \varphi ( 2^{-k-1} 2^{- j \beta''} \xi'') - \tilde \varphi ( 2^{-k} 2^{- j \beta''} \xi'')\right] \hat f(\xi',\xi''), \\
 (Q_j  f)^\wedge (\xi',\xi'') & := \tilde \varphi(2^{-j \beta''} \xi'') \hat f(\xi',\xi'').
\end{align*}
Observe that for fixed $k$, the operators $P_{jk}$ exhibit a scaling symmetry consistent with the homogeneity condition, but that for fixed $j$, the scaling is isotropic (and, hence, potentially conflicting).  Observe that $|\xi''|_{- j \beta''} \leq 1$ in the frequency support of $Q_j$ and $2^{k-1} \leq |\xi''|_{-j \beta''} \leq 2^{k+1}$ for $P_{jk}$, and that, for each $j$, the sum
\[ Q_j + \sum_{k=0}^\infty P_{jk} = I\]
where $I$ is the identity operator.  As with the $T_j$'s, this equation can be interpreted as saying $Q_j f + \sum_{j=0}^\infty P_{jk} f = f$  for any Schwartz function $f$ supported a finite distance away from the hyperplane $y' = 0$.  In this case, the convergence is in the Schwartz topology, and every term $Q_j f$ and $P_{jk} f$ retains the property that it is supported away from $y' = 0$.  

The main decomposition of the operator $T$, then,  will be the following sum over $j$ and $k$:
\begin{equation} T = \sum_{j=0}^\infty T_j Q_j + \sum_{j=0}^\infty \sum_{k=0}^\infty T_j P_{jk}. \label{decomp}
\end{equation}
At one point, it will also be necessary to use the summation-by-parts equality
\[\sum_{j=0}^{\infty} T_j Q_j = U_0 Q_0 + \sum_{j=1}^{\infty} U_j (Q_j - Q_{j-1}). \]
In the {\it a priori} sense, this equality is valid because of the finite summation-by-parts formula
\[\sum_{j=0}^{N} T_j Q_j = U_0 Q_0 - U_{N+1} Q_N +  \sum_{j=1}^{N} U_j (Q_j - Q_{j-1}) \]
coupled with the fact that $U_{N+1} Q_N f = 0$ for $N$ sufficiently large when $f$ is supported away from $y' = 0$.

Lastly, each of these decompositions remains valid (i.e., is defined in the {\it a priori} sense) if a fractional differentiation operator is applied on one or both sides (though if $J_{\gamma_R}^{s_R}$ is applied on the right, the test function $f$ must be chosen so that $J_{\gamma_R}^{s_R} f$ is supported away from $y'=0$ rather than $f$ itself).
\section{``Trivial'' inequalities}

This section contains the proofs of a variety of inequalities typically referred to as ``size'' or ``trivial'' inequalities, the reason being that the proofs of these inequalities typically do not depend on the geometry of $S$ in any real way, only on the size of the support of the cutoffs involed.  Of course, when fractional differentiations are added to the mix (as will be done shortly), oscillatory integral estimates and integration-by-parts arugments like lemma \ref{ibplemma} are necessary to establish even the trivial inequalities.

Before making this addition, though, it is necessary and worthwile to make a series of straightforward estimates which are not especially subtle in any way.  In light of the decomposition \eqref{decomp}, the indices $j$ and $k$ will be fixed from this point and through the next several sections to refer exclusively to the indices of summation in \eqref{decomp}.  Moreover, the following notation is adopted:  the expression $A \lesssim B$ will mean that there exists a constant $C$ such that, for all $j,k \geq 0$, $A \leq CB$ (and so $A \lesssim B$ is only meaningful if one or both sides depend on either $j$ or $k$).  If the expression $A$ or $B$ includes a fractional integration, the expression $A \lesssim B$ mean that $A \leq C B$ uniformly in $j$, $k$, and the imaginary parts of any fractional integration exponents.

With this notational device in hand, the first and most basic set of inequalities to establish is the following:
\begin{align}
||T_j Q_j  ||_{1 \rightarrow 1} & \lesssim 2^{-j |\alpha'|}, \label{trivial1} \\
||T_j P_{jk} ||_{1 \rightarrow 1} & \lesssim 2^{-j |\alpha'|}, \label{trivial2} \\
||T_j Q_j  ||_{\infty \rightarrow \infty} & \lesssim 2^{-j |\beta'|},  \label{trivial3} \\
||T_j P_{jk} ||_{\infty \rightarrow \infty} & \lesssim 2^{-j |\beta'|}, \label{trivial4} \\
||T_j Q_j  ||_{1 \rightarrow \infty} & \lesssim 2^{j |\beta''|}, \label{trivial5}\\
||T_j P_{jk} ||_{1 \rightarrow \infty} & \lesssim 2^{j |\beta''|+ kn''}. \label{trivial6} 
\end{align}
The unifying theme of these inequalities is that they are proved fairly directly from estimates of the size of the support of the amplitude $\psi_j$ appearing in the definition of $T_j$.  In fact, 
\begin{align*}
\int  |T_j f(x)| dx & \leq
 \int \! \! \! \int |f(y',x''+S(x,y')| |\psi_j(x,y')| dy' dx \\
& \leq \int \! \! \! \int \left( \int |f(y',x'')| dx'' \right) \sup_{x''} |\psi_j(x,y')| dx' dy' &
\lesssim 2^{-j |\alpha'|} ||f||_1
\end{align*}
since $2^{-j |\alpha'|}$ represents the size of the support of $\sup_{x''} |\psi_j(x,y')|$ in $x'$ (for fixed $y'$).  Similar reasoning gives that $||T_j||_{\infty \rightarrow \infty} \lesssim 2^{- j |\beta'|}$.  The Littlewood-Payley-type projections $Q_j$ and $P_{jk}$ are uniformly bounded on $L^p$ for all $p$ (since each $Q_j$ can be appropriately rescaled to $Q_0$ and each $P_{jk}$ to $P_{00}$); thus \eqref{trivial1}--\eqref{trivial4} follow.

The main observation behind the $L^1$-$L^\infty$ inequality is that Fubini's theorem guarantees that the following inequalities hold uniformly in $j$ and $k$:
\begin{align*} 
||Q_j||_{L^1 \rightarrow L^1_{y'} L^\infty_{y''}} & \lesssim 2^{j |\beta''|},  \\
||P_{jk}||_{L^1 \rightarrow L^1_{y'} L^\infty_{y''}} & \lesssim 2^{j |\beta''| + k n''}. 
\end{align*}
The justification for these estimates is that both $Q_j$ and $P_{jk}$ can be expressed as a convolution with a measure of smooth density on the hyperplane $x' = 0$.  The density is bounded by a constant times $2^{j |\beta''|}$ in the former case and $2^{j |\beta''| + kn''}$ in the latter, which can be seen by simply rescaling the operators $Q_j$ and $P_{jk}$ to coincide with $Q_0$ and $P_{00}$ as before.  From these facts and the definition of $T_j$, however,
\[ |T_j P_{jk}  f(x)| \lesssim 2^{j |\beta''| + k n''}  \int |\psi_j(x,y')| \int_{\R^{n''}} |f(y',y'')| dy'' dy' \lesssim 2^{j \beta'' + k n''} ||f||_1 \]
(and likewise for $T_j Q_j$).

\subsection{Fractional differentiation and $L^\infty$-$L^\infty_{x''} BMO^{\alpha'}_{x'}$ bounds}
In order to prove theorem \ref{sobolevthm}, it is absolutely essential to prove generalizations of \eqref{trivial2} and \eqref{trivial3} in the presence of fractional derivative operators (and \eqref{trivial1} and \eqref{trivial2} as well, but these are readily obtained from what is known about the dual operator $T^*$).  Moreover, to obtain a range of sharp results, it is necessary here just as in the work of Christ, Nagel, Stein, and Wainger \cite{cnsw1999} 
to be able to sum the corresponding estimates in a critical case (here, when there is no decay in $j$ of the norms of the individual terms).  For this reason, stating the inequality as an $L^1$-$L^1$ or $L^\infty$-$L^\infty$ bound is unsatsifactory; even the Calder\'on-Zygmund weak-$(1,1)$ bound is unsuccessful here (unlike in \cite{cnsw1999}) because its proof requires that a separate $L^p$-$L^p$ has already been established.  In general, the operators here are expected to be bounded on $L^p$ for a single value of $p$ in the critical case (because the rate of decay varies as $p$ varies unlike the translation-invariant case in which it is constant).

The solution is to directly prove a BMO-type inequality and appeal to analytic interpolation.  In this case, the operators in question may not even be bounded from $L^\infty$ to BMO, but they are bounded from $L^\infty$ to a mixed-norm space involving $L^\infty$ and a nonisotropic version of BMO.  The space will be designated $L^\infty_{x''} BMO^{\alpha'}_{x'}$, and is defined to be the space of functions $f$ for which there exists a constant $C_f$ such that, for almost every $x''$ and any box $B$ on $\R^{n'}$ with side lengths $2^{s \alpha'_i}$ for $i = 1,\ldots,n'$ ($s \in \R$),
\[\frac{1}{|B|} \int_B |f(x) - \left<f\right>_{B,x''}| dx' \leq C_f \]
where $ \left<f\right>_{B,x''} := \frac{1}{|B|} \int_B f(x',x'') dx'$.
 The inequality to be proved in this section, then, is that when $\real{s_L}, \real{s_R} \geq 0$ and $\real{s_L} \frac{\tilde \alpha}{\gamma_L} + \real{s_R} \frac{\beta}{\gamma_R} = |\beta'|$, then for every fixed $\epsilon > 0$,
\begin{align}  \left| \left| J_{\gamma_L}^{s_L} \left( \sum_{j=0}^\infty T_j Q_j \right) J_{\gamma_R}^{s_R} \right| \right|_{L^\infty \rightarrow L^\infty_{x''} BMO^{\alpha'}_{x'}}  & \lesssim  1, \label{bmoq} \\
\left| \left| J_{\gamma_L}^{s_L} \left( \sum_{j=0}^\infty T_j P_{jk} \right) J_{\gamma_R}^{s_R} \right| \right|_{L^\infty \rightarrow L^\infty_{x''} BMO^{\alpha'}_{x'}}  & \lesssim  2^{k (\real{s_L} \frac{\one}{\gamma_L} + \real{s_R} \frac{\one}{\gamma_R}+\epsilon)}, \label{bmop} 
\end{align}
uniformly in $k$, $\imaginary{s_L}$ and $\imaginary{s_R}$ (recall that $\frac{\delta}{\gamma} := \max_i \frac{\delta_i}{\gamma_i}$).
  
To prove \eqref{bmoq} and \eqref{bmop}, it is first necessary to revisit the ``trivial'' inequalities in the presence of fractional differentiation, as well as to introduce several new inequalities:
\begin{align}
 ||J_{\gamma_L}^{s_L} T_j Q_{j} J_{\gamma_R}^{s_R} ||_{\infty \rightarrow \infty} & \lesssim 1, \label{bmoineq1}\\
 ||J_{\gamma_L}^{s_L} T_j P_{jk} J_{\gamma_R}^{s_R} ||_{\infty \rightarrow \infty} & \lesssim 2^{ k ( \real{s_L} \frac{\one}{\gamma_L}  +  \real{s_R} \frac{\one}{\gamma_R})},
\label{bmoineq2} \\
 ||J_{\gamma_L}^{s_L} (\partial_{- j \alpha'}')^l T_j Q_{j} J_{\gamma_R}^{s_R} ||_{\infty \rightarrow \infty} & \lesssim 1,  \label{bmoineq3} \\
 ||J_{\gamma_L}^{s_L} (\partial_{-j \alpha'-k \one}')^l  T_j P_{jk} J_{\gamma_R}^{s_R} ||_{\infty \rightarrow \infty} & \lesssim 
2^{ k ( \real{s_L} \frac{\one}{\gamma_L}  +  \real{s_R} \frac{\one}{\gamma_R})},
 \label{bmoineq4} \\
||J_{\gamma_L}^{s_L} T_j Q_{j} J_{\gamma_R}^{s_R}  ||_{\infty \rightarrow L^\infty_{x''} L^1_{x'}} & \lesssim 2^{- j |\alpha'|}, \label{bmoineq5} \\
||J_{\gamma_L}^{s_L} T_j P_{jk} J_{\gamma_R}^{s_R} ||_{\infty \rightarrow L^\infty_{x''} L^1_{x'}} & \lesssim 2^{- j |\alpha'| + k ( \real{s_L} \frac{\one}{\gamma_L}  +  \real{s_R} \frac{\one}{\gamma_R})},
 \label{bmoineq6}
\end{align}
where $(\partial'_{-j \alpha'})^l$ represents a scaled, mixed derivative in only the single-primed directions (i.e., not in the double-primed directions).
The proofs of these inequalities are virtually identical because it will not be necessary to use the fact that $P_{jk}$ is cutoff away from small frequencies, which is the main qualitative feature distinguishing it from $Q_j$.  For this reason, the attention will be focused primarily on \eqref{bmoineq2}, \eqref{bmoineq4}, and \eqref{bmoineq6}.  In what follows, for the proofs of \eqref{bmoineq1}, \eqref{bmoineq3}, and \eqref{bmoineq5}, simply fix $k = 0$.

By \eqref{diff2}, it suffices to prove a modified form of \eqref{bmoineq1}-\eqref{bmoineq6}.  Specifically, it suffices to replace $J_{\gamma_L}^{s_L}$ by $J_{\gamma_L}^{s_L} |_{j \tilde \alpha + k \one}$ and $J_{\gamma_R}^{s_R}$ by $J_{\gamma_R}^{s_R} |_{j \beta + k \one}$; this is true by virtue of the identity
\begin{equation} J_{\gamma_L}^{s_L} f = \left.  J_{\gamma_L}^{s_L} \right|_{j \tilde \alpha + k \one} f + \left(  J_{\gamma_L}^{s_L} - \left.  J_{\gamma_L}^{s_L} \right|_{j \tilde \alpha + k \one} \right)  \left. J_{\gamma_L}^0 \right|_{j \tilde \alpha + k \one} f \label{restrid}
\end{equation}
(and likewise for $J_{\gamma_R}^{s_R}$) which is itself true because $\left. J_{\gamma_L}^0 \right|_{j \tilde \alpha + k \one}$ is the identity operator.  Therefore, one may assume without loss of generality that differential inequalities of the form \eqref{diff1} hold (which will be necessary to apply lemma \ref{ibplemma}). 

For convenience, let $V_{jk} f(x) := J_{\gamma_L}^{s_L}|_{j \tilde \alpha + k \one} T_j P_{jk} J_{\gamma_R}^{s_R}|_{j \beta + k \one} f(x)$.
The function $V_{jk} f(x)$ is given by integration against a kernel $K_{jk}(x,y)$, given by the expression
\begin{equation}
 \int e^{2 \pi i \left(\xi \cdot(x-w) + \eta' \cdot(z'-y') + \eta'' \cdot (w'' + S(w,z') - y'') \right)} \varphi_{jk}(\xi,\eta,w,z') d \xi d \eta dw dz', \label{bmoker1}
\end{equation}
where the amplitude function $\varphi_{jk}$ is equal to the product of several simpler pieces:  the cutoff $\psi_j(w,z')$ from the definition of $T_j$; the cutoff in $\eta''$ arising from $P_{jk}$, which happens to be supported on the set where $|\eta''|_{-j \beta''} \leq 2^{k+1}$;  and finally, the Fourier transforms $(J_{\gamma_L}^{s_L}|_{j \tilde \alpha + k \one})^\wedge(\xi)$ and $(J_{\gamma_R}^{s_R}|_{j \beta + k \one})^\wedge(\eta)$.  In the case of \eqref{bmoineq3} and \eqref{bmoineq4}, the amplitude that arises is slightly different.  This time the amplitude is given by
\begin{equation}
e^{- 2 \pi i \eta'' \cdot S(w,z')} (\partial_{- j \alpha' - k \one}')^l \left[ e^{2 \pi i \eta'' \cdot S(w,z')} \varphi_{jk} ( \xi, \eta,w,z') \right] \label{bmoamp2}
\end{equation}
with the derivative acting on the $w'$ variables only.  This new amplitude can, of course, be expressed as a finite linear combination of scaled $w'$-derivatives of $\varphi_{jk}$ times a finite number of scaled $w'$-derivatives of $\eta'' \cdot S(w,z')$ (a simple integration-by-parts is all that is necessary to turn the derivative in $x'$ to a derivative in $w'$).

At this point, the main piece of information needed to apply lemma \ref{ibplemma} is the scale $\scaled$ to be used.  To that end, choose scale ${-j \alpha'} - k \one$ in the $w'$ variable, $j \alpha' + k \one$ in the $\xi'$ variable, and $- j \beta' - k \one$ and $j \beta' + k \one$, respectively, in $z'$ and $\eta'$.  In the remaining directions, the scales chosen are $j \beta'' + k \one$ in $\xi''$ and $\eta''$ and $- j \beta'' - k \one$ in $w''$.  With respect to the chosen scale, all scaled derivatives of degree at least two of the phase \[\Phi_{x,y}(\xi,\eta,w,z') := 2 \pi  \left(\xi \cdot(x-w) + \eta' \cdot(z'-y') + \eta'' \cdot (w'' + S(w,z') - y'') \right)\] are bounded uniformly in $j$ and $k$; that is, for all $j,k$, $|\partial^{l}_S \Phi_{x,y} (\xi,\eta,w,z')| \leq C_{l}$ when $|l| \geq 2$.  This fact is a direct consequence of the uniform convergence of the scaled derivatives of $S$, as in \eqref{limits}, coupled with the fact that $\beta''_i > \alpha''_i$ for all $i$.  Likewise, the scaled derivatives of the cutoff $\varphi_{jk}$ are all uniformly bounded in $j$ and $k$ (and the imaginary parts of $s_L$ and $s_R$) by a constant times
\begin{equation} \left( 2^{\real{s_L} (j \frac{\tilde \alpha}{\gamma_L}+ k \frac{\one}{\gamma_L})} + \sum_{i=1}^n |\xi_i|^{\frac{\real{s_L}}{(\gamma_L)_i}} \right) \left( 2^{\real{s_R} (j \frac{\beta}{\gamma_R} + k \frac{\one}{\gamma_R})} + \sum_{i=1}^n |\eta_i|^\frac{\real{s_R}}{(\gamma_R)_i} \right) \label{normnum}
\end{equation}
and supported where $|w'|_{j \alpha'} \lesssim 1$, $|z'|_{j \beta'} \lesssim 1$ and $|\eta''|_{- j \beta'' - k \one} \lesssim 1$.  Note that this fact is also true of the scaled derivatives of the amplitude \eqref{bmoamp2} since the scaled derivatives of the phase $\eta'' \cdot S(w,z')$ are uniformly bounded.

The magnitude of the scaled gradient of $\Phi_{x,y}$, on the other hand, is greater than some fixed constant (independent of $j$ and $k$) times
\begin{align*} 
2^k(|x' - w'|_{j \alpha'} & +  |x'' - w''|_{j \beta''} +  |z'-y'|_{j \beta'} + |w''+S(w,z)-y''|_{j \beta''}) \\
& + 2^{-k} (|\xi'- \nabla_{w'} \eta'' \cdot S(w,z') |_{-j \alpha'} +  |\eta'+ \nabla_{z'} \eta'' \cdot S(w,z') |_{-j \beta'}) \\
& + 2^{-k} |\xi'' - \eta'' - \nabla_{w''} \eta'' \cdot S(w,z')|_{-j \beta''}.
\end{align*}
Again, since the scaled derivates of $\eta'' \cdot S(w,z')$ are uniformly bounded, there is a constant $C_0$ independent of $j$ and $k$ such that the magnitude of the scaled gradent is greater than $C_0$ times
\begin{align*}
2^k(|x' - w'|_{j \alpha'} & +  |x'' - w''|_{j \beta''} +  |z'-y'|_{j \beta'} + |w''+S(w,z)-y''|_{j \beta''})  \\
& + 2^{-k} (|\xi' |_{-j \alpha'} +  |\eta'|_{-j \beta'} + |\xi'' - \eta''|_{-j \beta''})- C_0.
\end{align*}
Choose $\epsilon < C_0^{-1}$ and apply the integration-by-parts argument of lemma \ref{ibplemma}.  The result is that the kernel $K_{jk}(x,y)$ (modulo a multiplicative constant independent of $j$ and $k$) is bounded from above by the integral over $\xi, \eta, w$ and $z'$ (suitably cut-off in $w'$, $z'$ and $\eta''$) of a fraction whose numerator is \eqref{normnum} and whose denominator is 
\begin{equation}
\begin{split}
2^{kN}(|x' - w'|_{j \alpha'} & +  |x'' - w''|_{j \beta''} +  |z'-y'|_{j \beta'} + |w''+S(w,z)-y''|_{j \beta''})^N  \\
& + 2^{-kN} (|\xi' |_{-j \alpha'} +  |\eta'|_{-j \beta'} + |\xi'' - \eta''|_{-j \beta''})^N +1
\end{split} \label{normdenom}
\end{equation}
for any fixed positive integer $N$.
To obtain the operator norm of $V_{jk}$ on $L^\infty$, the kernel $K_{jk}(x,y)$ must be integrated over $y$ and the supremum over all $x$ is taken.  
This integral can be estimated by using proposition \ref{intlemma} recursively:  performing the $y$ integral first, lemma \ref{intlemma} dictates that the $L^\infty$ operator norm of $V_{jk}$ is less than the integral over $\xi,\eta,w,z'$ (suitably cutoff in $w'$, $z'$ and $\eta''$) of a new fraction whose numerator is \eqref{normnum} times an additional factor of $2^{j |\beta|}$, but whose denominator is (modulo a multiplicative constant independent of $j$ and $k$)
\begin{equation}
\begin{split}
2^{kN_2}(|x' - w'|_{j \alpha'} & +  |x'' - w''|_{j \beta''})^{N_2}  \\
& + 2^{-kN_2} (|\xi' |_{-j \alpha'} +  |\eta'|_{-j \beta'} + |\xi'' - \eta''|_{-j \beta''})^{N_2} +1
\end{split} 
\end{equation}
for $N_2 := N - n$.
Proposition \ref{intlemma} is repeated for the integrals over $w$, $\xi$, and $\eta'$ (in the process, the triangle inequality $|\xi_i| \leq |\xi_i - \eta_i| + |\eta_i|$ is used when the $\xi''$ integral is encountered to make terms in the numerator match terms in the denominator).  After these integrations are complete, the denominator is trivial (assuming that $N$ was chosen sufficiently large). To conclude, the integrals over $\eta''$ and $z'$ are estimated using the size of the support of $\varphi_{jk}$ in these directions.  Collecting all the powers of $2$ encountered in this way gives precisely the inequalities \eqref{bmoineq1} - \eqref{bmoineq4}.

For the norm of $V_{jk}$ as a mapping from $L^\infty$ to $L^\infty_{x''} L^1_{x'}$, the kernel $K_{jk}$ is integrated in $x'$ and $y$ and the supremum over $x''$ is taken.  Just as before, proposition \ref{intlemma} is applied recursively.  This time the order of integration is $y$ followed by $x'$, then $w'$, $\xi$ and $\eta'$.  After these steps, the denominator is again trivial, and the remaining integrals over $w'$, $z'$ and $\eta''$ are carried out by computing the size of the support of $\varphi_{jk}$ in these directions.   Collecting powers of $2$ as before gives precisely the same result as above with the addition of another factor of $2^{-j |\alpha'|}$.

In light of \eqref{bmoineq1}-\eqref{bmoineq6}, the argument to establish \eqref{bmoq} and \eqref{bmop} proceeds as follows.  First observe that given any smooth function $f$ on $\R^n$ and any box $B \subset \R^{n'}$ of side lengths $2^{t \alpha'_i}$ for $i=1,\ldots,n'$ and some $t \in \R$, the following inequality holds: 
\[ \frac{1}{|B|} \int_B |f(x) - \left<f\right>_{B,x''}| dx' \leq 2 \min \left\{ 
\frac{1}{|B|} ||f||_{L^\infty_{x''} L^1_{x'}}, ||f||_\infty, \sum_{l=1}^{n'}  2^{s \alpha_i'} \left|\left|\frac{\partial f}{\partial x_i'} \right|\right|_\infty
\right\}. \]
The first two terms on the right-hand side follow from fairly straightforward applications of the triangle inequality.  The latter perhaps requires more explanation.  The triangle inequality guarantees that
\[ \frac{1}{|B|} \int_B |f(x) - \left<f\right>_{B,x''}| dx' \leq \frac{1}{|B|^2} \int_B \int_B |f(x',x'')- f(y',x'')| dx' dy', \]
and the fundamental theorem of calculus allows one to estimate the difference $|f(x',x'')- f(y',x'')|$ in terms of the gradient:
\begin{align*} |f(x',x'')- f(y',x'')| & = \left| \int_0^1 \frac{d}{d \theta} f( \theta x' + (1-\theta) y',x'') d \theta \right| \\
& \leq \sum_{i=1}^{n'} \int_0^1 |x_i'-y_i'| \left| \left(\frac{\partial f}{\partial x_i'} \right) (\theta x' + (1- \theta) y',x'') \right| d \theta \\
& \leq 2 \sum_{i=1}^{n'} 2^{t \alpha_i'} \left|\left|\frac{\partial f}{\partial x_i'} \right|\right|_\infty.
\end{align*}
For any bounded $g$, let $f := J_{\gamma_L}^{s_L} \left( \sum_{j=0}^\infty T_j  \Qj \right) J_{\gamma_R}^{s_R} g$.  The inequalities \eqref{bmoineq1}, \eqref{bmoineq3}, and \eqref{bmoineq5} give that
\[\frac{1}{|B|} \int_B |f(x) - \left<f\right>_{B,x''}| dx' \lesssim \sum_{j=0}^\infty \min \left\{ 2^{(j- t) |\alpha'|},1,\sum_{i=1}^{n'} 2^{(t-j) \alpha_i} \right\} ||g||_\infty \]
uniformly in $t$ and $||g||_\infty$, of course.  Summing in $j$ and taking the supremum over $B$ and $x''$ gives \eqref{bmoq}.  If instead one takes $f := J_{\gamma_L}^{s_L} \left( \sum_{j=0}^\infty T_j  P_{jk} \right) J_{\gamma_R}^{s_R} g$, the same reasoning gives that
\begin{align*}
\frac{1}{|B|}  \int_B & |f(x) -  \left<f\right>_{B,x''}| dx' \\  \lesssim & \sum_{j=0}^\infty 2^{ k (\real{s_L} \frac{\one}{\gamma_L} + \real{s_R} \frac{\one}{\gamma_R})} \min \left\{ 2^{(j- t) |\alpha'|},1,\sum_{i=1}^{n'} 2^{(t-j) \alpha_i} 2^k \right\} ||g||_\infty, 
\end{align*}
which yields \eqref{bmop} (in fact, it yields the slightly better inequality in which $2^{\epsilon k}$ is replaced by $\log(1+k)$).
\section{$L^2$-$L^2$ inequalities}
\subsection{Orthogonality inequalities}
The goal of this section is to prove the necessary orthogonality inequalities for the operators $T_j Q_j$ and $T_j P_{jk}$ on $L^2(\R^n)$.  As in the previous section, a number of slightly different inequalities are necessary, but the proofs of these inequalities are nearly indistinguishable.  The precise statement of these inequalities goes as follows:  Fix $s_L, \gamma_L, s_R, \gamma_R$ and a positive integer $M$.  Then for any positive integers $j_1,j,k$, if $|j-j_1|$ is sufficiently large (independent of the choices of $j_1,j,$ and $k$) then
\begin{align}
 || P_{\! j_1 k} J_{\gamma_L}^{s_L} T_j J_{\gamma_R}^{s_R} P_{j k}||_{2 \rightarrow 2} & \lesssim 2^{-(k+j+j_1)M} \label{ortho1} \\
|| (Q_{j_1} - Q_{j_1-1}) J_{l_1}^{s_1} U_j J_{l_2}^{s_2} (Q_{j} - Q_{j-1})||_{2 \rightarrow 2} & \lesssim 2^{-(j + j_1) M}.  \label{ortho2}
\end{align}
If, in addition, $j$ is sufficiently large, then it is also true that
\begin{align}
 \left|\left| \left( Q_{0} + \sum_{k_1 = 0}^k P_{0k} \right) J_{\gamma_L}^{s_L} T_j J_{\gamma_R}^{s_R} P_{j k}\right|\right|_{2 \rightarrow 2} & \lesssim 2^{-(k+j)M} \label{ortho3} \\
|| Q_{0} J_{\gamma_L}^{s_L} U_j J_{\gamma_R}^{s_R} (Q_{j} - Q_{j-1})||_{2 \rightarrow 2} & \lesssim 2^{-j M}. \label{ortho4}
\end{align}
Heuristically speaking, these inequalities assert that $P_{jk}$ effectively commutes with $T_j$ and $Q_j - Q_{j-1}$ likewise effectively commutes with $U_j$, so that, for fixed $k$, the terms of the decomposition \eqref{decomp} are effectively mutually orthogonal.  The advantage of this, of course, is that it is precisely what is needed to apply the Cotlar-Stein almost-orthogonality lemma to conclude that the operator norm on $L^2$ of the sum (for fixed $k$) is comparable to the supremum of the operator norms over $j$ (which is an absolutely necessary element of the proof of theorem \ref{sobolevthm}).

The proof to be given now is that of \eqref{ortho1}; all others are proved in a similar manner.  Let $V_{jk} := J_{\gamma_L}^{s_L} T_j J_{\gamma_R}^{s_R} P_{jk}$.  Conjugated by the Fourier transform, the operator $V_{jk}$ has a kernel (on frequency space) given by 
\[K_{jk} (\xi,\eta) := \int e^{2 \pi i \left(-\xi \cdot w + \eta' \cdot z' + \eta'' \cdot (w'' + S(w,z')) \right)} \varphi_{jk} (\xi,\eta,w,z') d w dz'\]
where, as before, $\varphi_{jk}$ is supported where $|\eta''|_{-j \beta'' - k \one} \lesssim 1$, $|w|_{j \alpha} \lesssim 1$, and $|z'|_{j \beta'} \lesssim 1$; additionally,
\[ |\varphi_{jk}(\xi,\eta,w,z')| \lesssim \left( 1 + \sum_{i=1}^n |\xi_i|^\frac{\real{s_L}}{(\gamma_L)_i} \right)  \left( 1 + \sum_{i=1}^n |\eta_i|^\frac{\real{s_R}}{(\gamma_R)_i} \right). \]

Let $j_m := \max \{j_1,j\}$.  To estimate the size of the kernel $K_{jk}$, a suitable scale $\scaled$ must be chosen.
Choose scale $-j_m \alpha'- \frac{1}{2} k \one$ for $w'$ and $-j_m \beta'- \frac{1}{2} k \one$ for $z'$, then choose scale $- j_m \alpha'' - \frac{1}{2} k \one$ for $w''$.  The family of phases $\Phi_{\xi,\eta} (w,z')$ satisfies
\begin{align*} 
|\nabla \Phi_{\xi,\eta}(w,z')|_\scaled \gtrsim &  2^{-\frac{k}{2}} ( |\xi'|_{-j_m \alpha'} +  |\eta'|_{-j_m \beta'} +  |\xi''-\eta''|_{- j_m \alpha''}) - C_0 2^{\frac{k}{2}} 
\end{align*}
for some constant $C_0$ independent of $j,j_1$, and $k$ (due to the uniform convergence of the scaled derivatives $S$ as in \eqref{limits}).  The quantity $\epsilon_1 := \min_{i} \beta''_i - \alpha''_i$ is strictly positive;  clearly $|\xi'' - \eta''|_{-j_m \alpha''} \geq 2^{\epsilon_1 j_m} |\xi'' - \eta''|_{-j_m \beta''}$.  Now it must either be the case that $|\xi''|_{-j_m \beta''} \gtrsim 2^{k}$ or $|\eta''|_{-j_m \beta''} \gtrsim 2^k$.  If the former is true (which occurs when $j_m = j_1$), then $|\eta''|_{-j_m \beta''} \leq 2^{-\epsilon_2 |j -j_1|}|\eta''|_{-j \beta''} \lesssim 2^{-\epsilon_2 |j -j_1| + k}$, where $\epsilon_2 := \min_i \beta''_i$.  On the other hand, if $j_m = j$, then $|\xi''|_{-j_m \beta''} \lesssim 2^{- \epsilon_2 |j - j_1| + k}$ (this is the case which occurs in \eqref{ortho3} and \eqref{ortho4}). By the triangle inequality, then, $|\xi''-\eta''|_{-j_m \alpha''} \gtrsim 2^{\epsilon_1 j_m + k}$ when $|j-j_1|$ is sufficiently large (for some bound uniform in $j$ and $k$).  It follows that, when $|j-j_1|$ is sufficiently large, the scaled gradient of the phase satsifes the improved inequality
\begin{align*} 
|\nabla \Phi_{\xi,\eta}(w,z')|_\scaled \gtrsim &  2^{-\frac{k}{2}} ( |\xi'|_{-j_m \alpha'} +  |\eta'|_{-j_m \beta'} +  |\xi''-\eta''|_{- j_m \alpha''}) + 2^{\epsilon_1 j_m + \frac{k}{2}}. 
\end{align*}

To compute the operator norm on $L^1$ associated to the kernel $K_{jk}$, apply lemma \ref{ibplemma} (and note that the scaled derivatives of $\varphi_{jk}$ with respect to $w$ and $z'$ are clearly bounded when the cutoff arises from the operators $T_j$ or $U_j$), then integrate over $\xi$ and take the supremum over $\eta$.  As in the previous section, proposition \ref{intlemma} is applied to the integral in $\xi$.  Next, the integrals in $w$ and $z'$ are estimated using the size of the support of $\varphi_{jk}$.  The fact that the scaled gradient of the phase has magnitude no smaller than a constant times $2^{\epsilon_1 j_m + \frac{k}{2}}$  gives that the $L^1$-operator norm is less than a constant times $2^{-M j_m - M k}$ for any fixed positive $M$ by taking $N$ sufficiently large in lemma \ref{ibplemma}.

The operator norm on $L^\infty$ is computed in a completely analogous way, integrating over $\eta$ and taking the supremum over $\xi$; the result is the same, i.e., the operator norm on $L^\infty$ can be made smaller than $2^{-M j_m - M k}$ provided $|j - j_1|$ is sufficiently large.  Finally, Riesz-Thorin interpolation gives \eqref{ortho1}.

\subsection{van der Corput inequalities}

In this section, the rank condition on the mixed Hessian \eqref{hessian} finally comes into play.  Let $r$ be the minimum value of the rank of \eqref{hessian} over $(x',x'',y') \neq (0,0,0)$ and $\eta'' \neq 0$.  The main inequalities to be proved in this section are that for $j$ sufficiently large, for any $s_L, \gamma_L$, $s_R,\gamma_R$ and any $z$ satisfying
$\real{z} = \frac{|\alpha'| + |\beta'|}{2} - \real{s_L} \frac{\tilde \alpha}{\gamma_L} - \real{s_R} \frac{\beta}{\gamma_R}$, it must be the case that
\begin{align}
|| 2^{jz} J_{\gamma_L}^{s_L} T_j  J_{\gamma_R}^{s_R} Q_{j}||_{2 \rightarrow 2} \lesssim & 1 \label{vdcq} \\
|| 2^{jz} J_{\gamma_L}^{s_L} T_j  J_{\gamma_R}^{s_R} P_{jk}||_{2 \rightarrow 2} \lesssim & 2^{- k \frac{r}{2} + k (\real{s_L} \frac{\one}{\gamma_L} + \real{s_R} \frac{\one}{\gamma_R})} \label{vdcp}
\end{align}
uniformly in $j$, $k$, $\imaginary{s_L}$, $\imaginary{s_R}$, and (of course) $\imaginary{z}$.  As before, the inequality \eqref{diff2} and the identity \eqref{restrid} allow one to replace the fractional derivatives by $J_{\gamma_L}^{s_L}|_{j \tilde \alpha + k \one}$ and $J_{\gamma_R}^{s_R} |_{j \beta + k \one}$ (in the case of \eqref{vdcq}, take $k=0$).  Note that the condition that $j$ be sufficiently large is the same as requiring that the cutoff $\psi$ of the operator \eqref{theop} is supported sufficiently near the origin, and so has no major effect on the potency of any of these inequalities.

It is first necessary to further localize the cutoffs $\psi_{j}(x,y')$.  To that end, let $\varphi_1, \ldots, \varphi_m$ be any finite partition of unity on the support of $\psi_0$.  Define 
\[ T_j^i f(x) := \int f(y',x'' + S(x,y')) \psi_j(x,y') \varphi_i(2^{j \alpha} x, 2^{j \beta'} y') dy'. \]
The inequalities \eqref{vdcq} and \eqref{vdcp} will be proved with $T_j$ replaced by $T_j^i$ for $i=1,\ldots,m$, then summed over $i$ to obtain the estimates originally desired.  To simplify notation, the index $i$ will be supressed and it will simply be assumed that the cutoffs $\psi_j$ are sufficiently localized around the points $(2^{-j \alpha} x_0, 2^{-j \beta'} y_0')$.

As is customary, the engine behind the proof is a $TT^*$ argument; that is, the operator norm on $L^2$ of the operator
\[ \left. J_{\gamma_L}^{s_L} \right|_{j \tilde \alpha + k \one} T_j ( \left. J_{\gamma_R}^{s_R} \right|_{j \beta + k \one} P_{jk} )^2 T_j^* \left. J_{\gamma_L}^{s_L} \right|_{j \tilde \alpha + k \one} \]
(and likewise for $Q_j$) will be computed.  Just as in the previous proofs, this operator is given by integration against a kernel $K_{jk}(x,y)$ which is itself expressed as an oscillatory integral with phase $\Phi_{x,y}(\xi,\eta,\nu,w,z,u',v')$ given by
\[ 2 \pi ( \xi \cdot (x-w) + \eta \cdot (z - y) + \nu' \cdot u' + \nu'' \cdot ( w'' - z'' + S(w,u'+v') - S(z,v'))  ) \]
and amplitude $\varphi_{jk}(x,y,\eta,\xi,\nu,w,z,u',v')$ which is a product of these factors: $\psi_j(w,u'+v')$ and $\overline{\psi_j(z,u')}$ from the definition of $T_j$; $(J_{\gamma_L}^{s_L}|_{j \tilde \alpha + k \one})^\wedge(\xi)$ and $\overline{(J_{\gamma_L}^{s_L}|_{j \tilde \alpha + k \one})^\wedge(\eta)}$;  finally $|(J_{\gamma_R}^{s_R}|_{j \beta + k \one})^\wedge(\nu)|^2$ and the modulus squared of the cutoff arising from $P_{jk}$.

Once again, lemma \ref{ibplemma} will be the main computational tool once a suitable scale $\scaled$ is chosen.  Choose scale $-j \tilde \alpha - k \one$ for $w$ and $z$ and the dual scale $j \tilde \alpha + k \one$ for $\xi$ and $\eta$. Choose $j \beta'+ k \one$ for $\nu'$ and $j \beta'' + k \one$ for $\nu''$, and choose $-j \beta' - k \one$ for $u'$ and $-j \beta'$ for $v'$.  The thing to notice about this choice of scale is that the derivatives in $v'$ do {\it not} have a factor of $2^{-k}$ in the scale.  In the language of lemma \ref{ibplemma}, it is this special, asymmetric case which leads to operator van der Corput-type bounds for the kernel $K_{jk}$.  Of course, this omission also means that one must take some extra care in analyzing the scaled derivatives of the phase.

As before, one expects the scaled derivatives of order $2$ or greater of the phase are uniformly bounded {\it except} for those various derivatives are taken exclusively in the $\nu''$ and $v'$ directions (since all other derivatives have enough factors of $k$ to balance the fact that $|\nu''|_{-j \beta''} \approx 2^k$).  As for these exceptional derivatives, the relevant portion of the phase is examined by breaking it into two pieces.
The first is the difference $\nu'' \cdot (S(w,u'+v') - S(w,v'))$.  The fundamental theorem of calculus provides the identity
\begin{align*}
\nu'' \cdot (S(w,u'+v') & - S(w,v')) =  \\ & \sum_{i=1}^{n'} 2^k 2^{j \beta'_i} u_i' \int_0^1 (2^{-k} \nu'') \cdot 2^{- j \beta'_i} \frac{\partial S}{\partial y_i'} (w, \theta u' + v') d \theta; 
\end{align*}
it follows immediately from differentiating this equality that all scaled derivatives of $\nu'' \cdot (S(w,u'+v') - S(w,v'))$ are bounded uniformly by a constant times $|u'|_{j \beta' + k \one}$ since $w, u'$, and $v'$ are restricted to be suitably small.  The second piece to examine is $\nu'' \cdot (w'' - z'' + S(w,v') - S(z,v'))$.  Again, the fundamental theorem of calculus gives that the scaled derivative of this piece with respect to $\nu''_i$ is simply equal to
\begin{align*}
2^{j \beta''_i + k}  & (w''_i - z''_i) +\\
& \sum_{l'=1}^{n'} 2^{j \alpha'_{l'}+k} (w'_{l'} - z'_{l'}) \int_0^1 2^{-j \alpha'_{l'} + j \beta''_i} \frac{\partial S_i}{\partial x'_{l'}} (\theta w + (1-\theta)z,v') d \theta + \\
& \sum_{l''=1}^{n''} 2^{j \beta''_{l''} +k } (w''_{l''} - z''_{l''}) \int_0^1 2^{-j \beta''_{l''} + j \beta''_i} \frac{\partial S_i}{\partial x''_{l''}} (\theta w + (1-\theta)z,v') d \theta
\end{align*} 
where the integrals in the $l'$ sum are uniformly bounded and the integrals in the $l''$ sum tend to zero uniformly as $j \rightarrow \infty$ (because of the uniform convergence \eqref{limits} in both cases and the fact that the entries of $\beta''$ strictly dominate those of $\alpha''$).  Similarly, the scaled derivative of this second piece with respect to $v'_i$ is equal to
\begin{align*}
& \sum_{l'=1}^{n'} 2^{j \alpha_{l'}' + k}(w'_{l'} - z'_{l'}) \int_0^1 2^{-j \beta'_{i} -  j \alpha_{l'}'-k} \nu'' \cdot \frac{\partial^2 S}{\partial x'_{l'} y'_i} (\theta w + (1-\theta)z,v') d \theta + \\
& \sum_{l''=1}^{n''} 2^{j \beta''_{l''} + k} (w''_{l''} - z''_{l''}) \int_0^1 2^{-j \beta'_{i} - j \beta''_{l''} - k} \nu'' \cdot \frac{\partial^2 S}{\partial x''_{l''} y_i'} (\theta w + (1-\theta) z,v') d \theta.
\end{align*}
As before, the second integral tends uniformly to zero as $j \rightarrow \infty$ by virtue of \eqref{limits} and the domination of $\alpha''$ by $\beta''$.  The first integral is, in the limit, an average of the $(l',i)$-entry of the mixed Hessian matrix $H^P$ over points near some fixed point $(x_0',x_0'',y_0',\nu_0'')$ (without loss of generality, one may localize in $\nu''$ with a finite partition of unity as was already done for the physical variables).
Fixing $\Phi_2 := \nu'' \cdot (w'' - z'' + S(w,v') - S(z,v'))$, the information just given about the derivatives of this second term may be written in matrix form as
\begin{equation}
\left[ 
\begin{array}{c}
2^{j \beta + k \one } \partial_{\nu''}  \Phi_2 \\
2^{j \alpha'} \partial_{v'}  \Phi_2
\end{array}
\right] = \left[ \begin{tabular}{c|c}
A & B \\
\hline
C & D \\
\end{tabular} \right]
\left[ \begin{array}{c} 2^{j \alpha' + k \one}(w'-z') \\
2^{j \beta'' + k \one}(w'' - z'')
\end{array}
\right] \label{matrixeqn}
\end{equation}
where $A$ has uniformly bounded entries, $B$ tends uniformly to an $n'' \times n''$ identity matrix, $C$ tends uniformly to an integral of the rescaled Hessian matrix \eqref{hessian}, and $D$ tends uniformly to zero.
  The rank condition on $H^P$ implies that there is an $r \times r$ submatrix of $C$ which is invertible (with coeffiecients of the inverse bounded uniformly in $j$ and $k$).  For simplicity, assume that this submatrix lies in the first $r$ rows and $r$ columns of the full matrix $C$.  For $j$ sufficiently large, then, the matrix in \eqref{matrixeqn} has an $(r + n'') \times (r + n'')$-invertible submatrix (which must contain $B$).  It follows that for some uniform constant $C$,
\begin{align*}
 |\nabla_{\nu''}  \Phi_{x,y}|_{j \beta'' + k \one} + |\nabla_{v'} & \Phi_{x,y}|_{-j \beta'}  \gtrsim |w''- z''|_{j \beta'' + k \one} + \sum_{i=1}^r 2^{j \alpha'_i + k} |w'_i - z'_i| \\ & - C |u'|_{j \beta' + k \one} - C \sum_{i=r+1}^{n'} 2^{j \alpha'_i + k} |w'_i - z'_i|;
\end{align*}
furthermore, differentiating the identities for $\nabla_{\nu''}  \Phi_{x,y}$ and $\nabla_{v'} \Phi_{x,y}$ likewise gives that all scaled derivatives of the phase (of fixed order) are bounded uniformly by some constant times $1 + |u'|_{j \beta' + k \one} + |\nabla_{\nu''}  \Phi_{x,y}|_{j \beta'' + k \one} + |\nabla_{v'} \Phi_{x,y}|_{-j \beta'} + \sum_{i=r+1}^{n'} 2^{j \alpha'_i + k} |w'_i - z'_i|$.  The full scaled gradient, however, has magnitude at least
\begin{align*}
& 2^k ( |x' - w'|_{j \alpha'} + |x''-w''|_{j \beta''} + |z' - y'|_{j \alpha'} + |z'' - y''|_{j \beta''}) \\
+ & 2^{-k} ( |\xi'|_{-j \alpha'} + |\xi''-\nu''|_{-j \beta''} + |\eta'|_{-j \alpha'} + |\eta''-\nu''|_{-j \beta''}) \\
+ & 2^{k} |u'|_{j \beta'} + 2^{-k} |\nu'|_{-j \beta'} + 2^{k} |w'' - z'' + S(w,v') - S(z,v')|_{j \beta''} \\
+ & |\nabla_{v'} \nu'' \cdot [S(w,v') - S(z,v')]|_{-j \beta'} - C.
\end{align*}
Restrict attention for the moment to the situation in which $|x_i' - y_i'|\leq 2^{-k - j \alpha_i'}$ for $i = r+1,\ldots,n'$.  In this case 
\[ 2^{j \alpha'_i + k} |w_i' - z_i'| \leq  2^{j \alpha'_i + k} |x_i' - w_i'| +  2^{j \alpha'_i + k} |y_i' - z_i'| + 1\]
for $i = r+1, \ldots, n'$.
Thus if one decreases the scales of $\nu''$ and $v'$ to equal $j \beta'' + (k-m)\one$ for $\nu''$ and $- j \beta' - m \one$, respectively, for some fixed $m$ suitably large (independent of $j$ and $k$ and the imaginary parts of $s_L$ and $s_R$), it follows that all scaled derivatives of the phase have magnitude at most $1 + |\nabla \Phi_{x,y}|_\scaled$ (up to a uniform multiple) and that the magitude of the scaled gradient is at least
\begin{align*}
& 2^k ( |x' - w'|_{j \alpha'} + |x''-w''|_{j \beta''} + |z' - y'|_{j \alpha'} + |z'' - y''|_{j \beta''}) \\
+ & 2^{-k} ( |\xi'|_{-j \alpha'} + |\xi''-\nu''|_{-j \beta''} + |\eta'|_{-j \alpha'} + |\eta''-\nu''|_{-j \beta''}) \\
+ & 2^{k} |u'|_{j \beta'} + 2^{-k} |\nu'|_{-j \beta'} +  2^{k} |w'' - z''|_{j \beta''} + 2^k |w'-z'|_{j \alpha'} - C.
\end{align*}

Now apply lemma \ref{ibplemma} and proposition \ref{intlemma} recursively as before.  Since the operator in question is self-adjoint, it suffices to compute its norm as a mapping on $L^1$, meaning that the kernel $K_{jk}$ should be integrated over $x$ and the supremum should be taken over $y$.  The integration over $x$, performed first, gives a factor of $2^{-j |\tilde \alpha| - k n}$ and reduces the denominator to
\begin{align*}
& 2^k( |z' - y'|_{j \alpha'} + |z'' - y''|_{j \beta''}) \\
+ & 2^{-k} ( |\xi'|_{-j \alpha'} + |\xi''-\nu''|_{-j \beta''} + |\eta'|_{-j \alpha'} + |\eta''-\nu''|_{-j \beta''}) \\
+ & 2^{k} |u'|_{j \beta'} + 2^{-k} |\nu'|_{-j \beta'} +  2^{k} |w'' - z''|_{j \beta''} + 2^k |w'-z'|_{j \alpha'} + 1
\end{align*}
(taken to a suitably large power).
Integration over $w$ produces an additional factor of $2^{-j |\tilde \alpha| - k n}$ and eliminates the terms involving $w$ on the last line.  Integration over $\xi$ then over $\eta$ both give factors of $2^{-j |\tilde \alpha| - k n} 2^{\real{s_L} (j \tilde \alpha / \gamma_L + k \one / \gamma_L)}$ (because of the growth of the cutoff $\varphi_{jk}$).  Integration in $z$ gives yet another factor of $2^{-j |\tilde \alpha| - k n}$.  Over $u'$, one gets an additional $2^{-j |\beta'| - k n'}$.  Integration over $\nu$ (using the finite support in $\nu''$) gives a factor of $2^{j |\beta| + k n} 2^{2 \real{s_R}(j \beta/\gamma_R + k \one / \gamma_R)}$.  Lastly, the integral over $v'$ gives a factor of $2^{-j |\beta'|}$ because of its finite support.  Altogether, this gives an operator norm less than some uniform constant times $2^{-j \real{z} -k n' + 2 k (\real{s_L} \one/\gamma_L + \real{s_r} \one/\gamma_R)}$ (recalling the condition on $z$).

Recall, however, that this estimate is derived under the assumption that  $|x_i' - y_i'|\leq 2^{-k - j \alpha_i'}$ for $i = r+1,\ldots,n'$.  To acheive this condition, one must consider truncations of the form $\chi_l J_{\gamma_L}^{s_L} T_j P_{jk} J_{\gamma_R}^{s_R}$, where $\chi_l$ is a multiplication operator restricting $x_i$ to a suitably small interval.  Now
\[ \left|\left| \sum_l \chi_l J_{\gamma_L}^{s_L} T_j P_{jk} J_{\gamma_R}^{s_R}\right|\right|_{2 \rightarrow 2} \lesssim \left( \sum_l || \chi_l J_{\gamma_L}^{s_L} T_j P_{jk} J_{\gamma_R}^{s_R}||_{2 \rightarrow 2}^2 \right)^\frac{1}{2} \]
by orthogonality of the truncated operators (because the truncation can, of course, be performed in a locally finite way).  The sum over $l$ has at most $C 2^{k (n' -r)}$ terms, yielding the estimates \eqref{vdcq} and \eqref{vdcp}.
%
%
\subsection{Application of the Cotlar-Stein lemma}
At this point, the inequalities \eqref{ortho1}--\eqref{ortho4} can be combined with \eqref{vdcq} and \eqref{vdcp} to show that, when $\real{s_L}, \real{s_R} \geq 0$ and $\frac{|\alpha'|+ |\beta'|}{2} - \real{s_L} \frac{\tilde \alpha}{\gamma_L} - \real{s_R} \frac{\beta}{\gamma_R} = 0$,
\begin{align}
\left|\left|J_{\gamma_L}^{s_L} \left( \sum_{j=0}^\infty T_j Q_j \right) J_{\gamma_R}^{s_R} \right| \right|_{2 \rightarrow 2} & \lesssim 1, \label{full22q} \\
\left|\left|J_{\gamma_L}^{s_L} \left( \sum_{j=0}^\infty T_j P_{jk} \right) J_{\gamma_R}^{s_R} \right| \right|_{2 \rightarrow 2} & \lesssim 2^{-k \frac{r}{2} + k ( \real{s_L} \frac{\one}{\gamma_L} + \real{s_R} \frac{\one}{\gamma_R})}, \label{full22p}
\end{align}
uniformly in $k$ and $\imaginary{s_L}$ and $\imaginary{s_R}$.

The proof is simply an application of the Cotlar-Stein almost-orthogonality lemma.  Let $R_{jk} := J_{\gamma_L}^{s_L}  T_j P_{jk}  J_{\gamma_R}^{s_R}$.  The Littlewood-Payley-type projections $P_{jk}$ ensure that $R_{j_1 k} R_{j_2 k}^* = 0$ when $|j_1 - j_2|$ is greater than some fixed constant (because the frequency supports are disjoint).  On the other hand, 
\begin{align*}
R_{j_1 k}^* R_{j_2 k} = R_{j_1 k}^* \left( Q_0 + \sum_{k_3 = 0}^k P_{0k} \right) R_{j_2 k} + \sum_{j_3=0}^\infty R_{j_1 k}^* P_{j_3k} R_{j_2 k}
\end{align*}
By \eqref{ortho3}, the first term has operator norm at most equal to some constant times $2^{-M(j_1 + j_2 + k)}$ provided that $j_1$ and $j_2$ are sufficiently large (independent of $k$,$\imaginary{s_L}$, and $\imaginary{s_R}$).  Likewise, when $|j_1 - j_3|$ is sufficiently large, each term in the sum has norm at most $2^{-M(j_1 + j_3 + k)}$ by \eqref{ortho1}.  On the other hand, when $|j_2 - j_3|$ is sufficiently large, the terms have norm at most $2^{-M(j_2 - j_3 + k)}$.  In fact, when $|j_1 - j_2|$ is sufficiently large, {\it both} of the previous two cases must occur if either one occurs separately.  Thus, for any value of $j_2$, it must be the case that
\[ \sum_{j_1=0}^\infty ||R_{j_1 k}^* R_{j_2 k} ||_{2 \rightarrow 2} + ||R_{j_1 k}^* R_{j_2 k} ||_{2 \rightarrow 2} \lesssim 2^{2 k ( \real{s_L} \frac{\one}{\gamma_L} + \real{s_R} \frac{\one}{\gamma_R})} \]
uniformly in $j_2$, $k$, and the imaginary parts of $s_L$ and $s_R$.  This gives \eqref{full22p} by the Cotlar-Stein lemma.

The proof of \eqref{full22q} proceeds in essentially the same manner after a (crucial) summation by parts.  In particular, 
\[\sum_{j=0}^{\infty} T_j Q_j = \sum_{j=0}^{\infty} (U_j - U_{j+1}) Q_j  = U_0 Q_0 + \sum_{j=1}^{\infty} U_j (Q_j - Q_{j-1}). \]
Now the operator $J_{\gamma_L}^{s_L} U_{0} Q_0 J_{\gamma_R}^{s_R}$ is clearly bounded on $L^2$ uniformly in the imaginary parts of $s_L$ and $s_R$ (the argument does not differ from that of \eqref{bmoineq1}).  Now let
\[ R_j  := J_{\gamma_L}^{s_L} U_j (Q_j - Q_{j-1}) J_{\gamma_R}^{s_R}. \]
As before, $R_{j_1} R_{j_2}^* = 0$ when $|j_1 - j_2|$ is sufficiently large.  But the identity
\begin{align*}
 R_{j_1}^* R_{j_2} & = R_{j_1}^* Q_0 R_{j_2} + \sum_{j_3=1}^\infty R_{j_1}^* (Q_{j_3} - Q_{j_3-1}) R_{j_2}
\end{align*}
and the inequalities \eqref{ortho2} and \eqref{ortho4} guarantee that each term has operator norm rapidly decaying in both $|j_1 - j_3|$ and $|j_2 - j_3|$ when $|j_1 - j_2|$ is sufficiently large.

\section{Interpolation and summation}

\subsection{$L^p-L^q$ inequalities}

In this section, the inequalities \eqref{trivial1}-\eqref{trivial6} and \eqref{vdcp} are combined to obtain the promised $L^p$-improving estimates for the averaging operator \eqref{theop}.  The key is to establish the restricted weak-type estimates at the vertices of the appropriate polygon in the Riesz diagram, then interpolate with the Marcinkiewicz interpolation theorem.  

To begin, consider the operator $\sum_j T_j Q_j$.  Riesz-Thorin interpolation of \eqref{trivial1} and \eqref{trivial3} gives that $T_j Q_j$ is bounded on $L^p$ with an operator norm at most some fixed constant times $2^{-j|\alpha'|/p - j |\beta'|/p'}$ where $\frac{1}{p} + \frac{1}{p'} = 1$.  Choose any such $p$, and for simplicity, let $\theta := \frac{1}{p}$.  Now for any two measurable sets $E$ and $F$,
\begin{align*}
\left| \int \chi_F(x)  \sum_{j=0}^\infty T_j Q_j \chi_E (x) dx \right| & \leq \\ C \sum_{j=0}^\infty \min \{ & 2^{-j (\theta |\alpha'| + (1-\theta)|\beta'| )}  |E|^\theta |F|^{1 - \theta}, 2^{j |\beta''|} |E||F| \}
\end{align*}
by $L^p$-boundedness of $T_j Q_j$ as well as $L^1-L^\infty$ boundedness coming from \eqref{trivial5}.  Now there is a single value of $j$ (call it $j_0$ and note that $j_0$ possibly negative and amost assuredly not an integer) for which the two terms appearing in the minimum on the right-hand side are equal.  Away from this special value $j_0$, the minimum must decay geometrically with a ratio that is independent of $|E|$ and $|F|$.  Therefore the sum of all terms with $j > j_0$ is dominated by some constant times the size of the term with $j = j_0$, and likewise for the terms with $j \leq j_0$.  Solving the equation
$|E|^{1-\theta} |F|^\theta 2^{j_0(|\beta''| + \theta |\alpha'| + (1-\theta)|\beta'|)} = 1$ and substituting gives that
\[ \left|\int \chi_F(x)  \sum_{j=0}^\infty T_j Q_j \chi_E (x) dx \right| \leq C' |E|^{\frac{\theta |\tilde \alpha| + (1-\theta) |\beta'|}{ \theta |\tilde \alpha| + (1-\theta) |\beta|}} |F|^{1-\frac{\theta |\beta''|}{ \theta |\tilde \alpha| + (1-\theta) |\beta|}}. \]
From here, varying $\theta \in [0,1]$, using the Marcinkiewicz interpolation theorem, and doing some arithmetic give that $\sum_j T_j Q_j$ maps $L^p$ to $L^q$ provided that
\[ \frac{|\beta|}{p} - \frac{|\tilde \alpha|}{q} = |\beta'| \]
and $1 < p < q < \infty$.

As for the operator $\sum_{jk} T_j P_{jk}$, the procedure is in principle the same.  First of all, the inequalities \eqref{trivial2}, \eqref{trivial6}, and \eqref{vdcp} (with $s_L = s_R = 0$) give that
\begin{align*}
\left| \int \chi_F(x)  \sum_{k=0}^\infty \sum_{j=0}^\infty T_j P_{jk} \chi_E (x) dx \right| & \leq \\ C \sum_{k=0}^\infty \sum_{j=0}^\infty \min \{  2^{j |\beta''| + k n''} & |E||F|,  2^{-j |\alpha'|} |E|, 2^{-j \frac{|\alpha'| + |\beta'|}{2} - k \frac{r}{2}} |E|^\frac{1}{2} |F|^\frac{1}{2} \}.
\end{align*}

\begin{figure}
\centering
\begin{pspicture}(-0.5,-0.5)(6.5,6)
\psline[linewidth=1.5pt,arrows=<->](0,6.25)(0,0)(6.25,0)

\psline[linewidth=1pt](3,3)(0,4.5)
\psline[linewidth=1pt](3,3)(4.5,0)
\psline[linewidth=1pt](3,3)(6.25,4.625)

\psline[linewidth=0.25pt](3.5,2)(3.5,3.25)
\psline[linewidth=0.25pt](4,1)(4,3.5)
\psline[linewidth=0.25pt](4.5,0)(4.5,3.75)
\psline[linewidth=0.25pt](5,0)(5,4)
\psline[linewidth=0.25pt](5.5,0)(5.5,4.25)
\psline[linewidth=0.25pt](6,0)(6,4.5)

\psline[linewidth=0.25pt](3.5,3.25)(2.2,3.4)
\psline[linewidth=0.25pt](4,3.5)(1.4,3.8)
\psline[linewidth=0.25pt](4.5,3.75)(0.6,4.2)
\psline[linewidth=0.25pt](5,4)(0,4.577)
\psline[linewidth=0.25pt](5.5,4.25)(0,4.885)
\psline[linewidth=0.25pt](6,4.5)(0,5.193)
\psline[linewidth=0.25pt](6.25,4.836)(0,5.5)
\psline[linewidth=0.25pt](6.25,5.143)(0,5.807)
\psline[linewidth=0.25pt](6.25,5.45)(0,6.114)
\psline[linewidth=0.25pt](6.25,5.757)(1.609,6.25)
\psline[linewidth=0.25pt](6.25,6.064)(4.459,6.25)

\psline[linewidth=0.25pt](2.2,3.4)(3.5,2)
\psline[linewidth=0.25pt](1.4,3.8)(4,1)
\psline[linewidth=0.25pt](0.6,4.2)(4.5,0)
\psline[linewidth=0.25pt](0,4.385)(4.071,0)
\psline[linewidth=0.25pt](0,3.923)(3.642,0)
\psline[linewidth=0.25pt](0,3.461)(3.213,0)
\psline[linewidth=0.25pt](0,2.999)(2.784,0)
\psline[linewidth=0.25pt](0,2.537)(2.355,0)
\psline[linewidth=0.25pt](0,2.075)(1.926,0)
\psline[linewidth=0.25pt](0,1.613)(1.497,0)
\psline[linewidth=0.25pt](0,1.151)(1.068,0)
\psline[linewidth=0.25pt](0,0.689)(0.639,0)
\psline[linewidth=0.25pt](0,0.227)(0.21,0)

\put(2.9,-0.4){$\displaystyle j$}
\put(-0.3,3.1){$\displaystyle k$}
\put(1.62,1.75){I}
\put(4.62,1.75){II}
\put(2.8,4.26){III}
\end{pspicture}
\caption{The heavy lines indicate the regions in which one term is smaller than the other two; the finer lines indicate where the appropriate operator norm is constant.}
\label{thefig}
\end{figure}
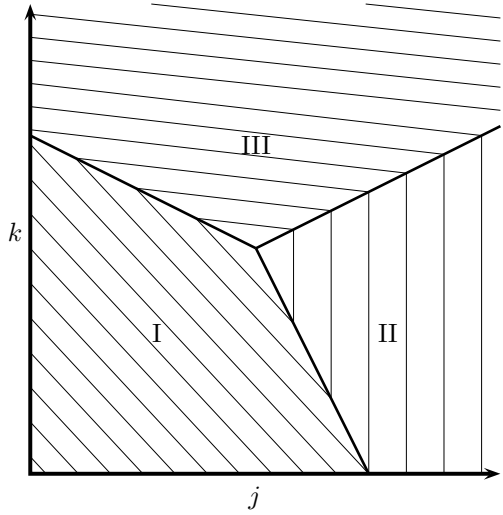

Now provided that $\frac{r}{n''} > \frac{|\alpha'| + |\beta'|}{|\beta''|}$, there is a unique pair of real numbers $j_0$ and $k_0$ at which the expression being summed attains a maximum.  See figure \ref{thefig} for a schematic illustration of the regions on which the first, second, and third term of the minimum, respectively, is the minimum.  Note that it is the condition on $\frac{r}{n''}$ which guarantees that the level lines of the operator norms (i.e., the lines where $j |\beta''| + k n''$ is constant in region I, $j |\alpha'|$ in region II, and $j \frac{|\alpha'| + |\beta'|}{2} - k \frac{r}{2}$ in region III) form closed triangles.  Now in each region, the operator norms decay geometrically as one moves away from $(j_0,k_0)$.  Furthermore, the number of terms of any fixed magnitude grows linearly with the distance from $(j_0,k_0)$.  Therefore, it is also true that the sum over all $j$ and $k$ is dominated by some constant times the value of the single term $j = j_0$, $k=k_0$.  At this particular point,
\[ 2^{j_0 |\tilde \alpha| + k_0 n''} |F| = 1 = 2^{j_0 \frac{|\alpha'| - |\beta'|}{2} - k_0 \frac{r}{2}} |E|^{- \frac{1}{2}} |F|^\frac{1}{2}; \]
solving gives $j_0 = \frac{ n'' \log_2 |E| - (n'' + r) \log_2 |F|}{|\tilde \alpha| r + (|\alpha'| - |\beta'|)n''}$ and $k_0 = \frac{-|\tilde \alpha|\log_2 |E| + |\beta| \log_2 |F|}{|\tilde \alpha| r + (|\alpha'| - |\beta'|) n''}$.  Substituting gives
\[ \left| \int \chi_F(x)  \sum_{k=0}^\infty \sum_{j=0}^\infty T_j P_{jk} \chi_E (x) dx \right| \leq C' |E|^{1- \frac{|\alpha'| n''}{|\tilde \alpha| r + (|\alpha'| - |\beta''|) n''}} |F|^\frac{|\alpha'| (n'' + r)}{|\tilde \alpha| r + (|\alpha'| - |\beta''|) n''}, \]
which gives precisely the vertex of the Riesz diagram circled in figure \ref{riesz} and lying above the line $\frac{1}{p} + \frac{1}{q} = 1$.  Performing the same procedure using \eqref{trivial4} for the second term instead of \eqref{trivial2} gives the second nontrivial vertex in figure \ref{riesz}.
\subsection{Sobolev inequalities}
To begin, observe that it suffices to replace the constraint \eqref{condition3} by the {\it a priori} stronger constraint that
\[ s \max \{ \alpha_1',\ldots,\alpha'_{n'},\beta''_1,\ldots,\beta''_{n''} \} \leq \frac{|\alpha'|}{p} + |\beta'|\left(1-\frac{1}{p} \right). \]
Suppose that $\alpha_j' > \beta_k''$.  Fix $\eta''_0 \in \R^{n''}$ to have $k$-th coordinate equal to $1$ and all other coordinates equal to zero; it follows from \eqref{homo1} and \eqref{limits} that the matrix $H^P(x',x'',y',\eta''_0)$ does not depend on $x_j'$.  Now fix $x'_0$ to have $j$'th coordinate equal to $1$ and all others zero.  The matrix $H^P(x'_0,0,0,\eta''_0) = H^P(0,0,0,\eta''_0)$ must have rank $r$, so there must be distinct indicies $l_1,\ldots,l_r$ and $m_1,\ldots,m_r$ (again distinct) such that $\alpha'_{l_1} + \beta'_{m_1} = \beta''_k$ and so on through $\alpha'_{l_r} + \beta'_{m_r} = \beta''_k$.  From this, it follows, however, that $\beta''_k < \frac{|\alpha'|+|\beta'|}{r}$.  Thus, if there were an $\alpha'_i$ greater than all entries of $\beta''$, the condition $\frac{r}{n''} > \frac{|\alpha'| + |\beta'|}{|\beta''|}$ could not hold.

Modulo this small change, theorem \ref{sobolevthm} follows somewhat more directly than do the $L^p$-$L^q$ inequalities.  Theorem 4 in chapter IV, section 5.2 of Stein \cite{steinha} is easily adapted to yield an analytic interpolation theorem for a an analytic family of operators $R_z$ where $R_{i \tau}$ maps $L^2$-$L^2$ (with operator norm bounded for all $\tau \in \R$) and $R_{1 + i \tau}$ maps $L^\infty$ to $L^\infty_{x''} BMO_{x'}^{\alpha'}$.  The key is to consider a partial sharp function of the form
\[ f^\sharp(x',x'') := \sup_B \int |f(x',x'') - \left<f\right>_{B,x''}| dx' \]
where $B$ ranges over all boxes in $\R^{n'}$ centered at $x'$ with appropriately nonisotropic side lengths.  The usual techniques (for example, a distributional inequality relating the sharp function to the associated maximal operator) demonstrate that
\[ \int |g(x',x'')|^p dx' \leq C_p \int |g^\sharp(x',x'')|^p dx' \]
(for a.e. $x''$) for some finite constant $C_p$ provided $p < \infty$; and Fubini's theorem guarantees that the coercivity inequality $||f||_p \leq C_p' ||f^\sharp||_p$ must hold for $p < \infty$ as well.  The linearization technique found in Stein \cite{steinha} now applies without further modification.  The result is that, for any fixed $2 \leq p < \infty$ and any real $s_L, s_R, \gamma_L,\gamma_R$ for which 
\[ \frac{|\alpha'|}{p} + |\beta'| \left( 1- \frac{1}{p} \right) = s_L \frac{\alpha'}{\gamma_L} + s_R\frac{\beta}{\gamma_R} \]
(taking $s_L$ and $s_R$ real) and any $\epsilon > 0$,
\begin{align*}
\left|\left| J_{\gamma_L}^{s_L} \left( \sum_{j=0}^\infty T_j Q_j \right) J_{\gamma_R}^{s_R} \right| \right|_{p \rightarrow p} & <  \infty\\
\left|\left| J_{\gamma_L}^{s_L} \left( \sum_{j=0}^\infty T_j Q_j \right) J_{\gamma_R}^{s_R} \right| \right|_{p \rightarrow p} & \lesssim  2^{k ( - \frac{r}{p} + s_L \frac{\one}{\gamma_L} + s_R \frac{\one}{\gamma_R} + \epsilon) }
\end{align*}
uniformly in $k$.  Fixing $s_R = 0$, for example, it follows that the fractional differentiation $J_{\gamma_L}^{s_L}$ applied to the sum \eqref{decomp} (summed over $j$ first, then $k$) converges in the strong operator topology provided that $\frac{r}{p} > s_L \frac{\one}{\gamma_L}$ and where $\frac{|\alpha'|}{p} + |\beta'| ( 1- \frac{1}{p} ) = s_L \frac{\alpha'}{\gamma_L}$.  Taking $\gamma_L = \one$ gives boundedness of $T$ from $L^p$ to $L^p_s$ for $p \geq 2$ as stated in theorem \ref{sobolevthm}.  The inequalities for $p \leq 2$ are proved by duality:  when $s_L = 0$ instead, fixing $\gamma_R = \one$ and $\frac{|\alpha'|}{p} + |\beta'| ( 1- \frac{1}{p} ) = s_R \frac{\alpha'}{\gamma_R}$, and $\frac{r}{p} > s_R \frac{\one}{\gamma_R}$ give that $T$ is bounded from $L^p_{-s_R}$ to $L^p$, so $T^*$ must map $L^{p'}$ to $L^{p'}_{s_R}$.  Since $T^*$ satisfies all the same homogeneity and rank conditions (with the roles of $\alpha'$ and $\beta'$ suitably interchanged), the portion of theorem \ref{sobolevthm} for $p \leq 2$ follows from this estimate just derived for dual operators $T^*$. 

\subsection{Necessity}
Necessity is shown by means of a Knapp-type example.  Consider the condition \eqref{condition1} first.  Let $E_s$ be a box in $\R^n$ with side lengths $2^{\beta_i t}$ for $i=1,\ldots,n$, and let $F_s$ be a box in $\R^n$ with side lengths $2^{\tilde \alpha_i t}$ (here $t$ is, of course, real).  Consider the integral
\[ \int \chi_{F_t}(x) T \chi_{E_t} (x) dx \]
For all $s$ sufficiently negative, the homogeneity conditions guarantee that the quantity $\chi_{E_s}(y',x''+S(x,y'))$ is identically one provided that $(x',x'') \in \epsilon E_t$ and $(y',x'') \in \epsilon F_t$ for some fixed constant $\epsilon > 0$ (here $\epsilon E_t$ is the set $E_t$ scaled linearly and isotropically down by a factor of $\epsilon$).  It follows that, when $\psi$ is greater than $\frac{1}{2}$ near the origin, one has
\[ \int \chi_{F_t}(x) T \chi_{E_t} (x) dx \geq \frac{1}{2} \epsilon^{2n' + n''} 2^{t(|\alpha'| + |\beta'| + |\beta''|)}, \]
and taking the limit as $t \rightarrow - \infty$, it follows that 
\[ \left| \int \chi_{F_t}(x) T \chi_{E_t} (x) dx \right| \leq C |E_t|^\frac{1}{p} |F_t|^{1 - \frac{1}{q} } \]
can hold uniformly for all $s$ only if $\frac{|\beta|}{p} - \frac{|\tilde \alpha|}{q} \leq |\beta'|$ (i.e., \eqref{condition1} must be satisfied for any appropriate choice of $S$).

As for condition \eqref{condition3}, standard arguments give that, when $T$ maps $L^p$ to $L^p_s$ for $1 < p < \infty$ and $s > 0$, one has
\[ ||P^\lambda_i T||_{p \rightarrow p} \leq C_p \lambda^{-s} \]
uniformly in $\lambda$, where 
\[(P^\lambda_i f)^\wedge(\xi) = \psi(\lambda^{-1} \xi''_i) \hat f(\xi) \]
for any smooth $\psi$ supported in $[-2,-1] \cup [1,2]$; choose $\psi$ to be nonnegative as well.  Now consider the integral
\[ \int (P^\lambda_i \chi_{F_t})(x) \int \chi_{E_t}(y',x''+ S(x,y')) \psi(x,y') dy' dx. \]
  Choosing $\lambda = \epsilon 2^{-t \beta''_i}$ for some fixed, small $\epsilon$, the function $(P^\lambda_i \chi_{F_t})(x)$ will be larger than some small constant $\epsilon'$ times the characteristic function $\chi_{F_t}(x)$ provided that $|x''_i| \leq 2^{t \beta''_i+1}$ (which is true of the support of $T \chi_{E_t}$ when $t$ is sufficiently negative).  Thus, Sobolev boundedness implies that 
\[ \epsilon' 2^{t (|\alpha'| + |\beta'| + |\beta''|)} \leq C' 2^{t s \beta''_i} 2^{|\beta|/p + |\tilde \alpha|(1-1/p)} \]
for all $t < 0$; letting $t \rightarrow - \infty$ and taking a supremum over $i$ gives \eqref{condition3}.

\section{Genericity considerations}

Suppose $M$ is an $n' \times n'$ matrix of rank $r$.  Transposing the order of rows and columns as necessary, it may be assumed that $M$ has the following block form:
\[ \left[
\begin{tabular}{c|c}
& \\
\hspace{5pt} A \hspace{5pt} & B \\
& \\
\hline
C & D\\
\end{tabular}
\right], \]
where $A$ is an $r \times r$ invertible submatrix, and $B$, $C$, and $D$ are $r \times (n'-r)$, $(n'-r) \times r$, and $(n'-r) \times (n'-r)$ submatrices, respectively.  The usual row-reduction arguments guarantee that 
$D - C A^{-1} B = 0$ for the matrix $M$.  Furthermore, this equation continues to be satisfied for all small perturbations of $M$ which are also rank $r$ matrices (where $A,B,C,D$ are, of course, replaced by their perturbations).  Suppose that ${\cal M}$ is some smooth mapping from a neighborhood of the origin in $\R^k$ into the space of $n' \times n'$ matrices such that ${\cal M}_0$ (that is, the matrix to which the origin maps) is equal to $M$.  The implicit function theorem, then, guarantees that the codimension (in $\R^k$) of the set of points near the origin mapping to a matrix of rank $r$ is at least equal to the rank of the differential of ${\cal M}$ at the origin minus $n'^2 - (n'-r)^2$.

Now let ${\cal P}^l_{\alpha',\alpha'',\beta'}$ be the space of polynomials $p$ in $x',x''$ and $y'$ (as always, $x', y' \in \R^{n'}$ and $x'' \in \R^{n''}$) for which $p(2^{\alpha'}x',2^{\alpha''}x'',2^{\beta'}y') = 2^l p(x',x'',y')$.  For conveinence, the subscripts $\alpha',\alpha'',$ and $\beta'$ will be supressed as these multiindices are considered ``fixed.''  Now given any multiindex $\beta''$, consider the following mapping from ${\cal P}^{\beta''_1} \times \cdots \times {\cal P}^{\beta''_{n''}} \times \R^{2n}$ into the space of $n' \times n'$ matrices given by
\begin{equation} (p_1,\ldots,p_{n''},x,y',\eta'') \mapsto \left( \left. \frac{\partial^2}{\partial x_i' y_j'} \right|_{x,y'} \sum_{k=1}^{n''} \eta''_k p_k \right)_{i,j=1,\ldots,n'}. \label{mapping}
\end{equation}
The goal is to compute the codimension in the space of ``pairings'' of polynomials and points, i.e., $(p_1,\ldots,p_{n''},x,y',\eta'')$, of those whose mixed Hessian has rank $r$.  In particular, if the codimension is large enough, then for a generic choice of polynomials $(p_1,\ldots,p_{n''})$ there will be no point (aside from the origin) at which th mixed Hessian has low rank.

To compute the rank of the differential of this map, it suffices by rescaling to assume that the coordinates of $x',y',x''$ and $\eta''$ are either $0$ or $1$; and of course one may assume that $\eta'' \neq 0$ and that at least one of $x'$, $x''$, or $'y'$ is also nonzero.  

Let $K_1$ be the least common multiple of all the entries of $\alpha', \alpha''$ and $\beta'$. Let $\Lambda$ be the set of positive integers $m$ which can be expressed as a sum $m = \alpha_i' + \beta_j'$ for some indices $i$ and $j$.  Now for any nonnegative integer $k$, 
\[\sum_{l \in \Lambda + k K_1} \# \set{(i,j)}{K_1 \mbox{ divides } l - \alpha_i' - \beta_j'} = (n')^2. \]
Fixing $K_2$ to be the cardinality of $\Lambda$, it follows that for at least one value of $l \in \Lambda + k K_1$, there are at least $K_2^{-1} (n')^2$ pairs of indices $(i,j)$ for which $K_1$ divides $l - \alpha_i' - \beta_j'$ (and therefore, $\alpha_m', \beta'_m$ and $\beta''_m$ divide this difference as well for all appropriate values of $m$).   It will now be shown that the rank of the differential of \eqref{mapping} is at least equal to $K_2^{-1} (n')^2$ provided that all the entries of $\beta''$ are congruent to some element of $\Lambda$ modulo $K_1$.

Suppose that $\beta''$ is as described, i.e., the entries of $\beta''$ are all congruent to some element of $\Lambda$ modulo $K_1$.  Suppose that $\eta''_{k_0} \neq 0$.  For any pair of indices $(i,j)$ such that $\beta''_{k_0} - \alpha_i' - \beta_j'$ is divisible by $k$, it must be the case that there is a monomial in ${\cal P}^{\beta''_{k_0}}$ of the form $x_l'^{p} x_i' y_j'$, $x_l''^{p} x_i' y_j'$ and $y_l'^p x_i' y_j'$ for any indices $l$ and appropriate values of $p$ in each case.  If $x_i'$ happens to be nonzero, then differentiating the $k_0$-th polynomial of \eqref{mapping} in the direction of the monomial $x_i'^{p+1} y_j'$ gives a matrix whose only nonzero entry is its $(i,j)$-entry.  Likewise, if $y_j'$ is nonzero, differentiation in the direction of $x_i' y_j'^{p+1}$ gives a matrix with only the $(i,j)$-entry nonzero.  Finally, if both $x_i'$ and $y_j'$ are zero, then differentiating in the direction of one of the remaining monomials $x_l'^{p} x_i' y_j'$, $x_l''^{p} x_i' y_j'$ or $y_l'^p x_i' y_j'$ for which $x_l'$, $x_l''$ or $y_l'$ is nonzero also gives a matrix with only the $(i,j)$-entry nonzero.

It follows that the codimension of points in ${\cal P}^{\beta''_1} \times \cdots \times {\cal P}^{\beta''_{n''}} \times \R^{2n}$ which have mixed Hessians of rank $r$ is at least $(n'-r)^2 - (1-K_2^{-1})(n')^2$ provided that the entries of $\beta''$ satisfy the congruence condition.  If this codimension is greater than $2n$, then it follows from projecting onto the space ${\cal P}^{\beta''_1} \times \cdots \times {\cal P}^{\beta''_{n''}}$ that such $n''$-tuples of polynomials generically have mixed Hessians with rank everywhere (except the origin) greater than $r$.
Thus, whenever $$r < n' - \sqrt{(1-K_2^{-1})(n')^2 + 2n},$$ the mixed Hessians \eqref{hessian} have rank everywhere equal to $r$ or greater (except at the origin).

\bibliography{mybib}

\end{document}